\documentclass[11pt]{amsart}
\usepackage{amscd}
\usepackage{xypic}  
\usepackage{amssymb}
\usepackage{amsthm}
\usepackage{epsfig}
\usepackage{paralist}
\usepackage{longtable}
\usepackage{lmodern}
\usepackage[T1]{fontenc}
\usepackage[utf8]{inputenc}
\usepackage{graphicx}

\usepackage{calc}

\usepackage[hidelinks=true]{hyperref}

\usepackage[normalem]{ulem}

\usepackage[all,cmtip]{xy}
\usepackage{amsmath}
\usepackage{xcolor}

\newtheorem{thm}{Theorem}[section] 

\newtheorem{lemma}[thm]{Lemma}
\newtheorem{prop}[thm]{Proposition}
\newtheorem{proposition}[thm]{Proposition}

\newtheorem{definition}[thm]{Definition}

\newtheorem{question}[thm]{Question}
\newtheorem{corollary}[thm]{Corollary}
\theoremstyle{definition}

\newtheorem{remark}[thm]{Remark}

 \newtheorem{notation}[thm]{Notation}
  \newtheorem{definition-remark}[thm]{Definition-Remark}
  \newtheorem{example}[thm]{Example}

\def\min{\operatorname{min}}

\def\max{\operatorname{max}}

\def\c1{\operatorname{c_1}}
\def\c2{\operatorname{c_2}}

\def\Sym{\operatorname{Sym}}

\def\ZZ{{\mathbb Z}}

\def\PP{{\mathbb P}}

\def\A{{\mathcal A}}

\def\SS{{\mathcal S}}
\def\G{{\mathcal G}}

\def\M{{\mathcal M}}

\def\O{{\mathcal O}}
\def\I{{\mathcal J}}

\def\E{{\mathcal E}}
\def\T{{\mathcal T}}

\def\F{{\mathcal F}}

\def\C{{\mathcal C}} 
\def\Q{{\mathcal Q}}

\def\c{\mathfrak{c}}

\def\x{\times}                   
\def\cong{\simeq}

\def\+{\oplus}               
\def\*{\otimes}                  

\def\mod{\operatorname{mod}}

\def\Pic{\operatorname{Pic}}

\def\Num{\operatorname{Num}}

\renewcommand{\leq}{\leqslant}
\renewcommand{\geq}{\geqslant}

\begin{document}

\title[Moduli spaces of polarized Enriques surfaces]
{Irreducible unirational and uniruled components of moduli spaces of polarized Enriques surfaces}

\author[C.~Ciliberto]{Ciro Ciliberto}
\address{Ciro Ciliberto, Dipartimento di Matematica, Universit\`a di Roma Tor Vergata, Via della Ricerca Scientifica, 00173 Roma, Italy}
\email{cilibert@mat.uniroma2.it}

\author[T.~Dedieu]{Thomas Dedieu}
\address{Thomas Dedieu,
Institut de Math\'ematiques de Toulouse~; UMR5219.
Universit\'e de Toulouse~; CNRS.
UPS IMT, F-31062 Toulouse Cedex 9, France} 
\email{thomas.dedieu@math.univ-toulouse.fr} 

\author[C.~Galati]{Concettina Galati}
\address{Concettina Galati, Dipartimento di Matematica e Informatica, Universit\`a della \linebreak Calabria, via P. Bucci, cubo 31B, 87036 Arcavacata di Rende (CS), Italy}
\email{galati@mat.unical.it}

\author[A.~L.~Knutsen]{Andreas Leopold Knutsen}
\address{Andreas Leopold Knutsen, Department of Mathematics, University of Bergen, Postboks 7800,
5020 Bergen, Norway}
\email{andreas.knutsen@math.uib.no}



\begin{abstract}
  We \color{black} prove that infinitely many irreducible components of the moduli space of
polarized Enriques  surfaces are
unirational (resp.\ uniruled), characterizing them in terms of decompositions of the
polarization as  an effective  sum of isotropic classes. \color{black}
In particular,  this applies to components of  
arbitrarily large genus $g$ and $\phi$-invariant of the polarization.
\end{abstract}

\maketitle


\section{Introduction} \label{sec:ointro}

Let $\E$ denote the  smooth  $10$-dimensional   moduli space
parametrizing smooth Enriques surfaces over  $\mathbb C$.  
A  polarized (resp.\ numerically polarized) Enriques surface
is a pair made of an Enriques surface together with an ample linear
(resp.\ numerical) equivalence class on it.
For integers $g>1$ and $\phi>0$, let 
$\E_{g,\phi}$ (resp., $\widehat{\E}_{g,\phi}$) denote the moduli
space of polarized (resp. numerically polarized) 
Enriques surfaces $(S,H)$ (resp. $(S,[H])$) 
such that 
$H^2=2g-2$ and $\phi(H)=\phi$, where
\begin{equation} \label{eq:defphi}
\phi(H):=\min \left\{E \cdot H \; | \; E^2=0, E > 0\right\}.
\end{equation}
Thus $g$ is the arithmetic genus of all curves in the linear system
$|H|$. There is an étale double cover $\rho: \E_{g,\phi} \to
\widehat{\E}_{g,\phi}$ mapping 
the two pairs
$(S,H)$ and $(S,H+K_S)$ to $(S,[H])$.   We refer to \S \ref{sec:gen00} for more details.  

The space $\E$ is irreducible and rational, as has been shown
by Kond{\=o} \cite{Ko},
and the forgetful maps $\E_{g,\phi} \to \E$ are étale.
Nevertheless
the spaces  $\E_{g,\phi}$ and $\widehat{\E}_{g,\phi}$ are in
general reducible,
and it is an open problem to \color{black} determine their Kodaira dimensions
\color{black} (cf.~\cite{dol}).

It is known that
$\E_{3,2}$ is irreducible and rational (cf.\ \cite{Ca}), that
$\E_{4,2}$ is irreducible and rational (this is the classical case of
\emph{Enriques sextics}, cf.\ \cite[\S 3]{dol}) and that $\E_{6,3}$
is irreducible and unirational (cf.\ \cite{Ve}),
and it  has  been
conjectured that the moduli spaces of polarized Enriques surfaces are
all unirational (or at least, of negative Kodaira dimension), 
see \cite[\S 4] {dol}.
\color{black}This was disproved by Gritsenko and Hulek in the recent paper  \cite{GrHu}, where the existence of \color{black} infinitely many irreducible components of general type
of the moduli 
space of  {\it numerically}  polarized Enriques surfaces \color{black} is established\color{black}. On the
other hand, they show that  all components of  
$\widehat{\E}_{g,\phi}$  have  negative Kodaira dimension  for  $g
\leqslant 17$.

In the present paper we  improve these  results.  Our interest lies in the moduli spaces of {\it polarized}  Enriques  surfaces $\E_{g,\phi}$:  
We give  a description  of  their  irreducible components  in  terms of decompositions of the
polarization as  an effective  sum of isotropic classes and prove
their unirationality (resp.\  uniruledness)
in infinitely many cases (for arbitrarily large $g$
and $\phi$).

To  explain our results,  we introduce  some notions.  Any effective line bundle $H$
with $H^ 2\geqslant 0$ on an Enriques surface
may be written as  (cf. Corollary \ref{cor:kl++} below) 
\begin {equation}\label{eq:ssid}  H \equiv a_1E_1+\cdots+a_nE_{n}\end{equation}
(where '$\equiv$' denotes numerical equivalence), such that  all $E_i$ are effective, non--zero, \emph{isotropic} (i.e., $E_i^2=0$) and \emph{primitive} (i.e., indivisible in $\Num(S)$), all $a_i$ are positive integers, $n \leqslant 10$ and  
\begin{equation}
\begin{cases}
 \mbox{either  $n \neq 9$,  $E_i \cdot E_j=1$ for all $i \neq j$,} \\
\label{eq:int2-int} 
\mbox{or  $n \neq 10$,  $E_1 \cdot E_2=2$ and $E_i \cdot E_j=1$ for all other indices
  $i \neq j$,} \\
 \mbox{or  $E_1 \cdot E_2=E_1 \cdot E_3=2$ and $E_i \cdot E_j=1$ for all
  other indices $i \neq j$,} 
\end{cases}
\end{equation}
up to reordering indices.
We call this a {\it simple isotropic decomposition},  cf. Definition \ref{def:tutte1}.  
An expression $H \sim a_1E_1+\cdots+a_nE_{n}+\varepsilon K_S$ 
(where '$\sim$' denotes linear equivalence) with $\varepsilon \in \{0,1\}$ satisfying the same conditions is also called a  simple isotropic decomposition.

We say that two polarized (respectively, numerically polarized) Enriques surfaces $(S,H)$ and $(S',H')$ in $\E_{g,\phi}$ (resp., $(S,[H])$ and $(S, [H'])$ in $\widehat{\E}_{g,\phi}$) {\em  admit the same  simple decomposition type}  (cf. Definition \ref{def:tutte2})  if one 
 has simple isotropic decompositions 
\begin{equation}\label{eq:sdt}  H \sim a_1 E_1+\cdots +a_nE_n+\varepsilon K_S
\; \; \mbox{and} \; \; H' \sim a_1 E'_1+\cdots +a_nE'_n+\varepsilon K_{S'}, \; \;  \mbox{with} \; \; \varepsilon   \in \{0,1\} \end{equation}
\[  \mbox {(resp.}\;\; H \equiv a_1 E_1+\cdots +a_nE_n
\; \; \mbox{and} \; \; H' \equiv a_1 E'_1+\cdots +a_nE'_n)\]
such that $E_i \cdot E_j=E'_i \cdot E'_j$ for all $i \neq j$.
We call $n$ the {\em length} of the decomposition  (type). 
  
If, possibly after reordering indices, there exists $r \leq n$ such
that $a_1=\cdots=a_r$ and $E_i \cdot E_j=1$ for 
 all $1 \leq i \leq r$ and $1 \leq j \leq n$, $i\neq j$,
 then we say that $(S,H)$ and $(S',H')$
{\em admit the same simple $r$-symmetric decomposition type}.

We note that in \eqref {eq:sdt} the case $\varepsilon=1$ is only needed when
all $a_i$'s are even, otherwise one may substitute any $E_i$ having odd
coefficient with $E_i+K_S$.  Also note that a given line
bundle may admit decompositions  of different types, cf. Remark \ref{rem:notunique},  but nevertheless the property of admitting the same decomposition type is an equivalence relation on $\E_{g,\phi}$ (and $\widehat{\E}_{g,\phi})$, cf. Proposition \ref{prop:trans}.

The main result of this paper is the identification of many
unirational (resp.\ uniruled) irreducible components of $\E_{g,\phi}$.
These are components of $\E_{g,\phi}$
parametrizing precisely those pairs $(S,H)$ with $H$  admitting the
same simple decomposition type.

\begin{thm} \label{thm:dom}
  The locus of pairs $(S,H) \in \E_{g,\phi}$  admitting the same  simple decomposition type of length $n\leqslant 4$ is an irreducible, unirational component of $\E_{g,\phi}$.

 The locus of pairs $(S,H) \in \E_{g,\phi}$  admitting the same  simple decomposition type of length $5$ is an irreducible component of $\E_{g,\phi}$, which is unirational if all $E_i \cdot E_j=1$ for all $i \neq j$, and uniruled otherwise. 
\end{thm}

\begin{thm} \label{thm:dom2}
  The locus of pairs $(S,H) \in \E_{g,\phi}$  admitting the same  simple $7$-symmetric (respectively, $6$-symmetric) decomposition type is an irreducible, unirational 
(resp., uniruled) component of $\E_{g,\phi}$.
\end{thm}

We  stress  that there are line bundles satisfying the assumptions of
these statements with arbitrarily large  $g$ and $\phi$. 
Moreover, there are decomposition types of all possible lengths $1\leq
n \leq 10$ to which these results apply.
For small values of 
$g$ or $\phi$ they actually provide all irreducible components of
$\E_{g,\phi}$, as stated in the following corollaries:

\begin{corollary} \label{cor:main1}
  When $\phi \leqslant 4$ the different  irreducible  components of $\E_{g,\phi}$ are precisely the loci  parametrizing pairs $(S,H)$ admitting the same  simple decomposition type   and they are all unirational. 
\end{corollary}

\begin{corollary} \label{cor:main2}
  When $g \leqslant 20$ the different  irreducible  components of $\E_{g,\phi}$ are precisely the loci  parametrizing pairs $(S,H)$  admitting the same simple decomposition 
type.  Moreover, they are all unirational, except possibly $\E_{16,5}$
and $\E_{17,5}$, which are 
in any event
irreducible and uniruled. 
\end{corollary}

As a further example, our results can also be used to describe the
 irreducible components of $\E_{g,\phi}$ for the highest values of
$\phi$ with respect to $g$, cf. Corollary \ref{cor:main3}.

  We note that the  proofs of our results 
 do not rely on the  construction of $\widehat{\E}_{g,\phi}$ in   \cite{GrHu}.

  By the above results  the (equivalence class of)  simple decomposition type seems to be the correct invariant to  distinguish  all the irreducible components of the moduli space of polarized Enriques
 surfaces. 
This is indeed true for numerical polarizations: we prove in 
Proposition \ref{prop:desccomp} that
the various  irreducible  components of $\widehat{\E}_{g,\phi}$ are
precisely the loci of pairs  admitting the same  simple decomposition
type.
 We do not know if the same holds for {\it linear} polarizations in
 full generality, cf.\ Question \ref{q:desccomp}\footnote
 {This question has subsequently been answered positively in  \cite[Thm.~1.1]{kn-JMPA}. }.

As another  application we   answer
\cite[Question 4.2]{GrHu} about
the irreducibility of the  preimage by
$\rho: \E_{g,\phi} \to \widehat{\E}_{g,\phi}$ of a component of
$\widehat{\E}_{g,\phi}$
under the assumptions of Theorems \ref{thm:dom}
and \ref{thm:dom2}:

\begin{corollary} \label{cor:conn}
Let $\C \subset \E_{g,\phi}$ be an irreducible component parametrizing
classes admitting the same simple decomposition type of length
$\leqslant 5$ or being $6$-symmetric.  Then $\rho^{-1}(\rho(\C))$ is
 reducible  if and only if $\C$ parametrizes  pairs $(S,H)$ such that $H$ is
$2$-divisible in $\Num(S)$. 
\end{corollary}

 Note that a class is $2$-divisible in $\Num(S)$ if and only if all coefficients in any simple isotropic decomposition are even, cf. Lemma \ref{lemma:H2div}.

 It is an interesting question whether this last corollary holds in general, that is, without any assumption on the decomposition types\footnote{This has subsequently been proved in  \cite[Thm.~4.2]{kn-JMPA},
   i.e.,
   Corollary~\ref{cor:conn} is valid in full generality, without any assumptions on
the decomposition types. \color{black} The proof of \cite[Thm.~4.2]{kn-JMPA} is different from our proof of Corollary~\ref{cor:conn}, but still relies on our results on simple isotropic decompositions in \S \ref{sec:bas1}.\color{black}}.

An interesting feature of our approach via simple isotropic
decompositions is that it enables one to 
write down  efficiently  the complete list
of all possible decompositions within a given numerical range.   (Note that  the datum of such a decomposition
prescribes of course the genus, but also the $\phi$-invariant,
cf.\ Remark  \ref{rem:calcphi}).  As an
illustration of our  methods we  catalogue all the irreducible
components of all  the  moduli spaces
$\widehat\E _{g,\phi}$ with $g \leq 30$ in an appendix;
for almost all of them we are able to determine the number of
corresponding irreducible components of $\E _{g,\phi}$, as well as
unirationality or uniruledness.
We do not make use of this list in
the present paper, but this information is needed in \color{black} our
paper \color{black} \cite{cdgk-mod} on moduli of curves on Enriques
surfaces. \color{black} Moreover, the approach to moduli spaces of polarized Enriques surfaces via
simple isotropic decompositions \color{black} also plays \color{black} a central role in our \color{black} subsequent \color{black} work  \cite{cdgk-sev} on Severi varieties on Enriques surfaces.

Our proofs of Theorems~\ref{thm:dom} and \ref{thm:dom2} are based on the
fact that a general Enriques surface has a model in $\PP^3$ as
an Enriques sextic, i.e.,  a
sextic surface singular along the six edges of a tetrahedron; such a
model corresponds to the datum of an isotropic sequence
$(E_1,E_2,E_3)$ with $E_i \cdot E_j=1$ for $i\neq j$, the $E_i$'s
corresponding to three edges of some face of the tetrahedron.
The idea is then to exhibit various irreducible and
rational (resp.\ uniruled) families $\mathcal{F}$ of elliptic curves in
$\PP^3$ with prescribed intersection numbers with the edges of some
fixed tetrahedron, such that a general Enriques sextic singular along
this particular tetrahedron contains a member of $\mathcal {F}$.
One thus gets incidence varieties that are irreducible and
rational (resp.\ uniruled) and dominate the corresponding components
of the moduli space  of polarized  Enriques  surfaces. This whole construction, which is very geometric in nature,  is done in \S \ref{sec:bas0}, where the proofs of our theorems and corollaries stated in this introduction are given; in  particular, Theorems \ref{thm:dom} and \ref{thm:dom2} are consequences of Propositions \ref{prop:dom} and
\ref{prop:dom2}.  Before this, in \S \ref{sec:bas1}, we prove the existence of simple isotropic decompositions  together with related technical results needed in \S \ref{sec:bas0}.


\vspace{0.3cm} 
\noindent
{\it Acknowledgements.} We thank Alessandro Verra for useful conversations on the  subject\color{black},  \color{black} Klaus Hulek for interesting correspondence about \cite{GrHu} and  \color{black} for \color{black} answering our questions
\color{black} and the referee for valuable comments. \color{black} We also acknowledge funding
from MIUR Excellence Department Project CUP E83C180 00100006 (CC),
project FOSICAV within the  EU  Horizon
2020 research and innovation programme under the Marie
Sk{\l}odowska-Curie grant agreement n.  652782 (CC, ThD),
 GNSAGA of INDAM (CC,CG), Trond Mohn
Foundation ``Pure mathematics in Norway'' (ThD, ALK) and grant n. 261756 of the Research Council of Norway (ALK).

\section{ Background results on moduli spaces} \label{sec:gen00}

 Let $\E$, $\E_{g,\phi}$ and  $\widehat{\E}_{g,\phi}$ be as in the introduction. The moduli space $\E$ is an open subset of a $10$-dimensional orthogonal modular variety, cf. \cite[{VIII \S 19-21}]{BHPV}. The moduli spaces $\E_{g,\phi}$ of polarized Enriques surfaces exist as quasi-projective varieties by \cite[Thm. 1.13]{Vie}. 

We  have   the  forgetful  map 
\begin{equation} \label{eq:forget}
\E_{g,\phi} \longrightarrow \E, 
\end{equation}
whose differential at a point $(S,H)$ is the linear map 
\begin{equation} \label{eq:forgetdiff}
 H^1(S,\E_H) \longrightarrow H^1(S,\T_S) 
\end{equation}
coming from the {\it Atiyah extension of $H$}
\[ 0 \longrightarrow \O_S \longrightarrow \E_H  \longrightarrow \T_S \longrightarrow 0, \]
by \cite[Prop. 3.3.12]{Ser}. Since $h^1(\O_S)=h^2(\O_S)=0$, the map \eqref{eq:forgetdiff} is an isomorphism, hence $\E_{g,\phi}$ is smooth and the map \eqref{eq:forget} is an \'etale cover.

The moduli spaces $\widehat{\E}_{g,\phi}$ exist by \cite{GrHu}. More precisely,
fixing an orbit  $\mathfrak{h}$ of the
action of the orthogonal group in the {\it Enriques lattice} $U \+ E_8(-1)$,
 in \cite{GrHu} the authors construct (irreducible) moduli spaces $\M^a_{\tiny{\mbox{En}},\mathfrak{h}}$ parametrizing isomorphism classes of numerically polarized Enriques surfaces $(S,[H])$ with $[H]$ in the orbit $ \mathfrak{h} \subset U \+ E_8(-1) \cong \Num(S)$ (see  \cite[Lemma VIII.15.1]{BHPV}). The spaces $\M^a_{\tiny{\mbox{En}},\mathfrak{h}}$ are open subsets of suitable orthogonal modular varieties.
Then our space $\widehat{\E}_{g,\phi}$ is the union of all $\M^a_{\tiny{\mbox{En}},\mathfrak{h}}$ where $\mathfrak{h}$ varies over all orbits 
with $\mathfrak{h}^2=2g-2$ and $\phi(\mathfrak{h})=\phi$, cf. \eqref{eq:defphi}. 
 It follows by \cite[Prop. 4.1]{GrHu} that there is an étale double cover $\rho: \E_{g,\phi} \to
\widehat{\E}_{g,\phi}$ mapping 
$(S,H)$ and $(S,H+K_S)$ to $(S,[H])$.

\section{ Background results  on line bundles on Enriques surfaces} \label{sec:gen0}

Any irreducible curve $C$ on an Enriques surface $S$ satisfies $C^2 \geqslant -2$, with equality if and only if $C$ is smooth and rational.  An Enriques surface containing such a curve is called {\it nodal}, otherwise it is called {\it unnodal}. On an unnodal Enriques surface, all divisors are nef and all divisors with positive self-intersection are ample.  It is well-known that the general Enriques surface is  unnodal, cf. references in \cite[p.~577]{cos2}.  

Recall that a divisor $E$ is said to be {\it isotropic} if $E^2=0$ and $E \not \equiv 0$. 
By Riemann-Roch, either $E$ or $-E$ is effective. It is said to be  {\it primitive}  if it is non-divisible in $\Num(S)$.  On an unnodal surface, any effective primitive isotropic divisor $E$ is represented by an irreducible curve of arithmetic genus one. 

 Let $H$ be an effective
line bundle with $H^2 >0$ and $\phi(H)$  as in \eqref {eq:defphi}. One has 
\begin{equation}
  \label{eq:bound}
  \phi(H)^2 \leqslant H^2, 
\end{equation}
by \cite[Cor. 2.7.1]{cd}, and there  are no cases satisfying $\phi(H)^2 < H^2 < \phi(H)^2 +\phi(H)-2$ by \cite[Prop. 1.4]{KL1}. Moreover \cite[Prop. 1.4]{KL1} also classifies the borderline cases as follows:

\begin{proposition} \label{prop:bound}
  Let $H$ be an effective line bundle on an Enriques surface satisfying $\phi(H)^2 \leqslant H^2 \leqslant \phi(H)^2 +\phi(H)-2$. Then one of the following occurs,  where $E_1,E_2,F$ are primitive, effective isotropic divisors satisfying $E_1 \cdot E_2=1$ and $E_1 \cdot F=E_2 \cdot F=2$:  
\begin{itemize}
\item[(i)] $H^2=\phi(H)^2$, in which case $H \equiv \frac{\phi(H)}{2}\left(E_1+F\right)$,
\item[(ii)] $H^2 =\phi(H)^2 +\phi(H)-2$, in which case,\\
\begin{inparaenum}
\item [$\bullet$] $H \sim \frac{\phi(H)-1}{2}(E_1+F)+E_2$ if $\phi(H)$ is odd, and \\
\item [$\bullet$] 
  $H \sim \frac{\phi(H)-2}{2}E_1+\frac{\phi(H)}{2}F+E_2$,
  or $H \equiv 2(E_1+E_2+F)$  \color {black}(and $\phi(H)=6$),
  \color{black} if $\phi(H)$ is even. 
\end{inparaenum}
\end{itemize}
\end{proposition}

 We recall the following from \cite[p.~122]{cd}:

\begin{definition} \label{def:rseq}
  An {\em isotropic $r$-sequence} on an Enriques surface $S$  is a sequence of isotropic effective divisors $\{E_1, \ldots, E_r\}$  such that $E_i \cdot E_j=1$ for $i \neq j$. 
\end{definition}

It is  well-known that any Enriques surface contains such sequences for every $r \leqslant 10$ \color{black} and that there are no such sequences with $r >10$ (cf. \cite[p.~175]{cd})\color{black}; moreover, by \cite[Cor. 2.5.6]{cd},  we have

\begin{proposition} \label{prop:cd}
Any isotropic $r$-sequence with $r \neq 9$ can be extended to a $10$-sequence.
\end{proposition}

 We will  also   make use of the following result:

\begin{lemma} \label{lemma:ceraprima}
(a) Let $\{E_1,\ldots,E_{10}\}$ be an isotropic $10$-sequence. Then there exists a divisor $D$ on $S$ such that $D^2=10$, $\phi(D)=3$ and
$3D \sim E_1+\cdots+E_{10}$. Furthermore, for any $i \neq j$, we have
\begin{equation} \label{eq:10-3}
  D \sim E_i+E_j+E_{i,j},
  \; \; \mbox{with}\; \; E_{i,j}^ 2=0
  \; \; \mbox{and}\; \; 
  E_{i,j}> 0.
\end{equation} 
\color{black} Hence, in particular,
$E_i \cdot E_{i,j}=E_j \cdot E_{i,j}=2$, and
\begin{equation}
  \label{eq:nuova}
  E_{i,j} \sim \frac{1}{3}\left(E_1+\cdots+E_{10}\right)-E_i-E_j,
\end{equation}\color{black}
 $E_k \cdot E_{i,j}=1 \; \; \mbox{for} \; \; k \neq i,j,$
\color{black} and \color{black}
$E_{i,j} \cdot E_{k,l}= \begin{cases} 1, \; \mbox{if} \; \{i,j\} \cap \{k,l\} \neq \emptyset, \\
2, \; \mbox{if} \; \{i,j\} \cap \{k,l\} = \emptyset. \end{cases}$ 

(b) Any divisor $D$ on $S$ such that $D^2=10$ and $\phi(D)=3$ satisfies $3D \sim E_1+\cdots+E_{10}$, for an isotropic $10$-sequence $\{E_1,\ldots,E_{10}\}$ consisting precisely of all isotropic divisors computing $\phi(D)$ up to numerical equivalence.  Moreover, if $F$ is a divisor satisfying $F^2=0$ and $F \cdot D=4$, then $F \equiv E_{i,j}$ for some $i \neq j$, where $E_{i,j}$ is defined by \color{black} \eqref{eq:nuova}\color{black}. 
\end{lemma}

\begin{proof} (a) The existence of $D$ is \cite[Lemma 1.6.2(i)]{cos2} or \cite[Cor. 2.5.5]{cd}. Its properties are easily checked and $E_{i,j}:=D-E_i-E_j$, cf. also \cite[Lemma 1.6.2(ii)]{cos2}.

(b) The first statement follows from \cite[Cor.~2.5.5]{cd} and its
proof.  For the last one, note that $F \cdot E_i >0$ 
for $i=1,\ldots, 10$ by
\cite[Lemma~2.1]{klvan}, whence, after permuting indices if necessary,
one gets
$F \cdot (E_1+ E_2)=4$ 
and $F \cdot E_i=1$ for $i=3,\ldots,10$.
Then $F \cdot E_{1,2}=0$ and $E_3 \cdot F=E_3
\cdot E_{1,2}=1$, so that $F \equiv E_{1,2}$ by \cite[Lemma
2.1]{klvan} again.
\end{proof}

\section{ Simple,  isotropic decompositions} \label{sec:bas1}

 One of the aims of this section  is to prove the existence of simple isotropic decompositions stated in the introduction (see Corollary \ref{cor:kl++}) and prove  that the isotropic divisors
occurring in such decompositions can always be extended
to an isotropic $10$-sequence plus one of the divisors $E_{i,j}$
occurring in Lemma \ref{lemma:ceraprima}  (see Corollary \ref{cor:kl+}).  The latter  will be needed in the
proof of our main results, see the comment right after
 Proposition  \ref{prop:dom2}.  We will also deduce several results on simple isotropic decompositions, like for instance the fact that $2$-divisibility can be read off any isotropic decomposition (see Lemma \ref{lemma:H2div})  and the fact that the property of admitting the same decomposition type as defined in the introduction is an equivalence relation on $\E_{g,\phi}$ and $\widehat{\E}_{g,\phi}$ (see Proposition \ref{prop:trans}).

 We start by recalling the  following from the introduction:

\begin{definition} \label{def:tutte1}
 Let $H$ be  an effective line bundle $H$
with $H^ 2\geqslant 0$ on an Enriques surface $S$. 

$\bullet$ An expression
$H \equiv a_1E_1+\cdots+a_nE_{n}$, where all  $a_i$ are positive integers, $n \leq 10$ and all
 $E_i$ are primitive, effective, isotropic divisors is called  a {\em simple isotropic decomposition} if \eqref{eq:int2-int} is satisfied, up to reordering indices.

$\bullet$ An expression  $H \sim a_1E_1+\cdots+a_nE_{n} + \varepsilon K_S$, with $\varepsilon \in \{0,1\}$, satisfying the same conditions will also be called a simple isotropic decomposition. 

$\bullet$ The number $n$ is the {\em length of the decomposition}.

$\bullet$ The decomposition is {\em $r$-symmetric} if, possibly after reordering indices, there exists $r \leq n$ such
that $a_1=\cdots=a_r$ and $E_i \cdot E_j=1$ for all $1 \leq i \leq r$ and $1 \leq j \leq n$, $i\neq j$  (equivalently, there is a set of $r$ isotropic divisors occurring in the decomposition with the same coefficient and each having intersection $1$ with the remaining isotropic divisors in the decomposition). 
\end{definition}

\begin{example} \label{ex:symm}
  Consider, in the notation of Lemma \ref{lemma:ceraprima}, the simple isotropic decomposition $H \equiv E_{1,2}+E_1+2E_2+E_3+E_4$. This is $2$--symmetric but not $3$-symmetric. Indeed, the set $\{E_3, E_4\}$ has the property that each member occurs in the decomposition with coefficient $1$ and intersects the remaining isotropic divisors in the decomposition in one point. There is no larger such set, since $E_1 \cdot E_{1,2}=2$ and $E_2$ occurs with coefficient $2$.
\end{example}

We recall \cite[Lemma 2.12]{KL1}:

\begin{lemma} \label{lemma:KL}
  Any effective line bundle $H$
with $H^ 2\geqslant 0$ on an Enriques surface can be written as
$H \equiv a_1E_1+\cdots+a_nE_{n}$,
where all  $a_i$ are positive integers, $1 \leq n \leq 10$, and all
$E_i$ are primitive, effective, isotropic divisors satisfying one of the  following three conditions:
\begin{itemize}
\item[(i)] $E_i \cdot E_j=1$ for all $i \neq j$,
\item[(ii)] $E_1 \cdot E_2=2$ and $E_i \cdot E_j=1$ for all other indices
  $i \neq j$, or
\item[(iii)] $E_1 \cdot E_2=E_1 \cdot E_3=2$ and $E_i \cdot E_j=1$ for all
  other indices $i \neq j$.
\end{itemize}
\end{lemma}

This lemma guarantees the existence of an effective decomposition satisfying almost all the conditions of a simple isotropic decomposition; indeed, what is missing, cf. \eqref{eq:int2-int},  is the additional requirement that $n \neq 9$ in case (i) and that
$n \neq 10$ in case (ii).

\begin{definition} \label{def:simple}
  A set $\{E_1,\ldots,E_n\}$ of primitive isotropic divisors on an Enriques
surface is called a {\em simple isotropic set} if it satisfies
 one of the conditions (i)-(iii) in Lemma \ref{lemma:KL}, possibly after permuting indices.

It is called a {\em  maximal simple isotropic set} if it is of the form 
$\{E_1,\ldots, E_{10}, E_{i,j}\}$, where $\{E_1,\ldots,E_{10}\}$ is an isotropic $10$-sequence and $E_{i,j}$ is defined up to numerical equivalence as in \color{black} \eqref{eq:nuova} \color{black} for some $i \neq j$.
\end{definition}

Note that since any simple isotropic set of $n$ elements contains
members of an isotropic $(n-1)$--sequence, any simple isotropic set
contains at most $11$ elements (cf. \cite[p.~179]{cd}).
\color{black} If it contains $11$ elements, then they necessarily satisfy (iii) in Lemma \ref{lemma:KL}, possibly after permuting indices. It will follow from Proposition \ref{prop:kl+} right below (cf. the footnote) that simple isotropic
sets of $11$ elements are precisely the maximal simple isotropic sets. 
\color{black} 
Also note that by \cite[Rem. p.~584]{cos2} any maximal simple isotropic set  generates  $\Num(S)$. 

 The following is a key  result, which generalizes Proposition \ref{prop:cd}. 

\begin{proposition} \label{prop:kl+}
 Any simple isotropic set  $\I$  can be extended to a maximal  simple isotropic set.\footnote{ This includes the case of simple isotropic sets of $11$ elements, which means that such are automatically maximal.}  Furthermore, if $\I =\{F_1,\ldots,F_n\}$ with $n \leq 9$, $F_1 \cdot F_2=2$ and $F_i \cdot F_j=1$ for $\{i,j\} \neq \{1,2\}$, then $\I$ can be extended to maximal simple isotropic sets such that either of $F_1$ or $F_2$ equals $E_{i,j}$. 
\end{proposition}

 We postpone the proof until the very end of the section to 
 discuss some consequences.  The first one yields the existence of simple, isotropic decompositions:

\begin{corollary}  \label{cor:kl++}
  Any effective line bundle $H$
with $H^ 2\geqslant 0$ on an Enriques surface has a simple isotropic decomposition.
\end{corollary}
 
\begin{proof}
 By Lemma \ref{lemma:KL}, we are done unless possibly if we end up in
case (i) with $n=9$ or 
in case (ii) with $n=10$. We treat these two cases separately and prove that in both cases we will find a different isotropic decomposition of $H$ satisfying condition (iii) of Lemma \ref{lemma:KL}, thus being a simple isotropic decomposition as desired. 

Assume first that $H \equiv a_1E_1+\cdots+a_{9}E_{9}$ with all $E_i \cdot E_j=1$ for $i \neq j$. Recalling Proposition \ref{prop:cd}, we divide the treatment into the two cases:
\begin{itemize}
\item[(a)] $\{E_1,\ldots,E_9\}$ can be extended to an isotropic $10$-sequence $\{E_1,\ldots,E_{10}\}$;
\item[(b)] $\{E_1,\ldots,E_9\}$ cannot be extended to an isotropic $10$-sequence $\{E_1,\ldots,E_{10}\}$.
\end{itemize}

In case (a), there is by Lemma \ref{lemma:ceraprima}(a) a primitive effective
isotropic $E_{9,10}$ such that 
$3E_{9,10}+2E_{9}+2E_{10} \equiv E_1+\cdots+E_8$.
Let $m:=\min_{1 \leq i \leq 8}\{a_i\}$. Then 
\[ H \equiv 3mE_{9,10}+(2m+a_9)E_{9}+2mE_{10} + (a_1-m)E_1+\cdots+(a_{8}-m)E_{8},\] where all $a_i-m \geq 0$ for $1 \leq i \leq 8$, at  least one being zero. Thus, the latter decomposition satisfies condition (iii) of Lemma \ref{lemma:KL}.

In case (b), then, by Proposition \ref{prop:kl+}, the set $\{E_1,\ldots,E_9\}$ can be extended to a maximal simple isotropic set $\{E_1,\ldots,E_9,E_{10},E_{11}\}$. 
This set contains an isotropic $10$-sequence by definition, which cannot contain
$\{E_1,\ldots,E_9\}$ by assumption. Possibly after reordering indices, we may thus assume that $\{E_2,\ldots,E_9,E_{10},E_{11}\}$ is an isotropic $10$-sequence, $E_1 \cdot E_{10}=E_1 \cdot E_{11}=2$ and $E_1 \equiv \frac{1}{3}(E_2+\cdots+E_{11})-E_{10}-E_{11}$, equivalently $3E_1+2E_{10}+2E_{11} \equiv E_2+\cdots+E_9$.
Let $m:=\min_{2 \leq i \leq 9}\{a_i\}$. Then 
\[ H \equiv (a_1+3m)E_1+2mE_{10}+2mE_{11} + (a_2-m)E_2+\cdots+(a_{9}-m)E_{9},\] where all $a_i-m \geq 0$ for $2 \leq i \leq 9$, at  least one being zero. Thus, the latter decomposition satisfies condition (iii) of Lemma \ref{lemma:KL}.

Assume next that $H \equiv a_1E_1+\cdots+a_{10}E_{10}$ with $E_1 \cdot E_2=2$ and $E_i \cdot E_j=1$ for all other indices
  $i \neq j$. By Proposition \ref{prop:kl+}, the set $\{E_1,\ldots,E_{10}\}$ can be extended to a maximal simple isotropic set $\{E_1,\ldots,E_{10},E_{11}\}$.
Possibly after interchanging $E_1$ and $E_2$, we may assume that $E_1 \cdot E_{11}=2$, $E_i \cdot E_{11}=1$ for $i \geq 2$ and 
$E_1 \equiv \frac{1}{3}(E_2+\cdots+E_{11})-E_2-E_{11}$, equivalently
  $3E_1+2E_2+2E_{11} \equiv E_3+\cdots+E_{10}$.
Let $m:=\min_{3 \leq i \leq 10}\{a_i\}$. Then 
\[ H \equiv (a_1+3m)E_1+(a_2+2m)E_2+2mE_{11} + (a_3-m)E_3+\cdots+(a_{10}-m)E_{10},\] where all $a_i-m \geq 0$ for $i \geq 3$, at  least one being zero. Thus, the latter decomposition satisfies condition (iii) of Lemma \ref{lemma:KL}.
\end{proof}

The next result yields a ``canonical'' way of writing any simple isotropic decomposition in $\Pic(S)$, which will be central in our proofs.

\begin{corollary} \label{cor:kl+}
  Let $H$ be any effective divisor on an Enriques surface such that $H^2 >0$. Then there is an isotropic $10$-sequence $\{E_1,\ldots,E_{10}\}$ (depending on $H$) such that
\begin{equation} \label{eq:kl+}
 H  \sim  a_0E_{1,2}+a_1E_1 + \cdots + a_{10}E_{10}+\varepsilon K_S, 
\end{equation}
where $E_{1,2}  \sim  \frac{1}{3}\left(E_1+\cdots+E_{10}\right)-E_1-E_2$ (cf. \color{black} \eqref{eq:nuova}\color{black}) and $a_0,a_1, \ldots, a_{10}$
are nonnegative integers  with

\begin{equation}
\begin{cases}
 \mbox{either $a_0=0$ and $\#\{i \; | \; i \in \{1,\ldots,10\}, a_i>0\} \neq 9$,} \\
\label{eq:conda010} 
\mbox{or $a_{10}=0$,} 
\end{cases}
\end{equation}
 and

\begin{equation} \label{eq:eps}
\varepsilon= \begin{cases} 0, & \mbox{if some $a_i$ is odd} \\
0 \; \mbox{or} \; 1, & \mbox{if all $a_i$ are even.} \
\end{cases}
\end{equation} 
 In particular, \eqref{eq:kl+} is a simple isotropic decomposition.

More precisely,  given any simple isotropic decomposition 
$H \equiv b_1F_1+\cdots+b_nF_{n}$,
we may find an expression \eqref{eq:kl+} such that each $F_i$ occurs in it (up to numerical equivalence) with coefficient $b_i$ (and the remaining coefficients in \eqref{eq:kl+} are zero).  Moreover, if $F_i \cdot F_j=2$ for only one pair of indices $i,j$, then we may find isotropic $10$-sequences satisfying either of the conditions $F_i \equiv E_{1,2}$ and $F_j \equiv E_{1,2}$. 
\end{corollary}

\begin{proof}
  Let $H \equiv b_1F_1+\cdots+b_nF_{n}$ be a simple isotropic decomposition, $n \leq 10$.  By Proposition \ref{prop:kl+}  there is  a maximal
simple isotropic set  $\{E_1,\ldots,E_{10},E_{1,2}\}$ containing the set $\{F_1,\ldots,F_{n}\}$.  Moreover, if $F_i \cdot F_j=2$ for only one pair of indices $i,j$, then $n \leq 9$ by definition of a simple isotropic decomposition,
so that we can make sure, still by Proposition \ref{prop:kl+}, that either of $F_i$ or $F_j$ equals $E_{1,2}$. 
 Thus, we may write
 $H \equiv a_0E_{1,2}+a_1E_1 + \cdots + a_{10}E_{10}$, where each $F_i$ occurs  (up to numerical equivalence) with coefficient $b_i$, the other coefficients are zero and
where $E_{1,2} \equiv
\frac{1}{3}\left(E_1+\cdots+E_{10}\right)-E_1-E_2$.
\color{black}
This gives an expression of $H$ as in \eqref{eq:kl+} with $\varepsilon \in \{0,1\}$.
Note that
\color{black}
$a_i=0$ for at least one $i$.
We may furthermore by symmetry assume that
\[
a_1 \geq a_2 \; \; \mbox{and} \; \; a_3 \geq \cdots \geq a_{10}.
\]

We claim that either $a_0=0$ or $a_{10}=0$. Indeed, if
$a_0 >0$ and $a_{10}>0$, then we must have $a_2=0$. If $a_1=0$, then the length of the decomposition is $9$ and this contradicts the first line in \eqref{eq:int2-int}. If $a_1>0$, then the length of the decomposition is $10$ and this contradicts the second line in \eqref{eq:int2-int}. This proves our claim.  Moreover, if $a_0=0$,  the first line in \eqref{eq:int2-int} implies that $\#\{i \; | \; i \in \{1,\ldots,10\}, a_i>0\} \neq 9$. \color{black} Thus,  \eqref{eq:conda010} is satisfied. \color{black}

  If 
not all $a_i$ are even, we can, possibly after replacing one $E_i$ having odd coefficient $a_i$ by $E_i+K_S$, assume $\varepsilon=0$. \color{black} We can thus make sure that \eqref{eq:eps} is satisfied. \color{black} If $E_{1,2} \sim \frac{1}{3}\left(E_1+\cdots+E_{10}\right)-E_1-E_2$, we are done. If not,  we have
$3E_{1,2}+2E_1+2E_2 \sim E_3+\cdots+E_{10} +K_S$. If $a_0=0$ (respectively, $a_{10}=0$), we may replace $E_{1,2}$ by $E_{1,2}+K_S$ (resp., $E_{10}$ by $E_{10}+K_S$) without altering \eqref{eq:kl+}, and obtain the desired relation $3E_{1,2}+2E_1+2E_2 \sim E_3+\cdots+E_{10}$, that is, $E_{1,2} \sim \frac{1}{3}\left(E_1+\cdots+E_{10}\right)-E_1-E_2$.

One readily checks that \eqref{eq:kl+} under condition \eqref{eq:conda010} is a simple isotropic decomposition. 
\end{proof}

 The condition in \eqref{eq:eps} concerning the parity of the coefficients $a_i$ is related to divisibility properties of $H$, by the following:

\begin{lemma} \label{lemma:H2div}
  A line bundle $H$ on an Enriques surface  $S$  is numerically $2$-divisible  (that is, its class in $\Num(S)$ is $2$-divisible)  if and only if  all coefficients in any simple isotropic decomposition of $H$ in $\Num(S)$ are even. Furthermore, in this case, 
$(S,H)$ and $(S,H+K_S)$ belong to different irreducible components of the moduli space of polarized Enriques surfaces. 
\end{lemma}

\begin{proof}
  The if part  of the first assertion  is clear. To prove the converse,
assume that $H$ is numerically $2$-divisible and  let  
\[  H \equiv a_0E_{1,2}+a_1E_1+a_2E_2+a_3E_3+\cdots+a_{10}E_{10},\]
 be a simple isotropic decomposition in the form of  Corollary \ref{cor:kl+}  (modulo numerical equivalence); in particular, all $a_i \geq 0$  and $a_0=0$ or $a_{10}=0$. We consider these two cases separately and  let $E_{i,j}$ be defined as in 
\eqref{eq:10-3}.

Assume $a_0=0$. Since $(E_{i,j}-E_i) \cdot H=2a_i+a_j$, for $i \neq j$, and $H$ is numerically $2$-divisible, we must have all $a_j$ even, as desired. 

Assume $a_{10}=0$. For $i=1,2$ we have $(E_{i,10}-E_{10}) \cdot H=a_i$, hence $a_1$ and $a_2$ are even. For $i \geqslant 3$ we have $(E_i-E_{10}) \cdot H = -a_i$, hence also $a_i$ for $i \geqslant 3$ is even. Moreover $E_3\cdot H=a_0+a_1+a_2+a_4  +\cdots +a_{9}$,  and since $a_1,\ldots, a_{9}$  are all even, also $a_0$ is even.

 To prove the last assertion, assume, to get a contradiction, that
$(S,H)$ and $(S,H+K_S)$ belong to the same irreducible component of the moduli space of polarized Enriques surfaces. Then $H$ and $H+K_S$ are either both $2$-divisible in $\Pic(S)$ or not. However, we know that in the present case  only one of them is $2$-divisible, a contradiction.
\end{proof}

\begin{notation} \label{not:int}
  When writing a simple isotropic decomposition \eqref {eq:ssid} verifying \eqref {eq:int2-int} (up to  permutation of indices), we will usually adopt the convention that $E_i$, $E_j$, $E_{i,j}$ are primitive isotropic satisfying
$E_i \cdot E_j=1$ for $i \neq j$, $E_{i,j} \cdot E_i= E_{i,j} \cdot E_j=2$ and
$E_{i,j} \cdot E_k=1$ for $k \neq i,j$. This notation has already been used in Lemma \ref{lemma:ceraprima}  and Example \ref{ex:symm}.   (By Corollary \ref{cor:kl+}, there  is no ambiguity in this notation.) 
\end{notation}

\begin{remark} \label{ex:2divspec}
  The requirement that a simple isotropic decomposition satisfies 
$n \neq 9$ in case (i) and 
$n \neq 10$ in case (ii) of Lemma \ref{lemma:KL}, which is   equivalent to
 condition \eqref{eq:conda010},  is crucial in the last proof. Indeed, take 
$ H \equiv a_0E_{1,2}+a_1E_1+a_3E_3+\cdots+a_{10}E_{10}$, 
 where $a_1$ is an even nonnegative integer, and $a_0,a_3,\ldots,a_{10}$ are odd positive integers.  If $a_1=0$, then this decomposition is as in case (i) of Lemma \ref{lemma:KL} with $n=9$, and if $a_1>0$, then it is as in case (ii) with $n=10$. Hence, this is {\it not} a simple isotropic decoposition according to our definition. 
On the other hand $H$ is numerically $2$-divisible. Indeed, the claim is equivalent to $B:=E_{1,2}+E_3+\cdots+E_{10}$ being numerically $2$-divisible.
As
 \begin{eqnarray*}
   B & \equiv & 3(E_{1,2}+E_1+E_2) +(E_1+E_2+E_3+\cdots+E_{10})-2E_{1,2}-4E_1-4E_2 \\
     & \equiv &2 (E_1+\cdots+E_{10})-2E_{1,2}-4E_1-4E_2,
 \end{eqnarray*}
 using Lemma \ref{lemma:ceraprima}(a), the claim follows. 

 Note that by  Lemma \ref{lemma:H2div} we have that $(S,H)$ and $(S,H+K_S)$ belong to different irreducible components of the moduli space of polarized Enriques surfaces. 
\end{remark}

\begin{remark} \label{rem:eps}
  By Lemma \ref{lemma:H2div} we get that \eqref{eq:eps} is equivalent to
\begin{equation} \label{eq:eps2}
\varepsilon= \begin{cases} 0, & \mbox{if $H+K_S$ is not $2$-divisible in $\Pic(S)$,} \\
1, & \mbox{if $H+K_S$ is $2$-divisible in $\Pic(S)$.} 
\end{cases}
\end{equation} 
This means that the '$\varepsilon$' in expression \eqref{eq:eps} only depends on $H$ and not on the simple isotropic decomposition.
\end{remark}

\begin{remark} \label{rem:calcphi}
 Writing a simple isotropic decomposition of $H$ as in \eqref{eq:kl+} has the advantage that $\phi(H)$ is calculated by one among $E_{1,2},E_1,\ldots,E_{10}$. More precisely, setting $a:=\sum_{i=0}^{10}a_i$, one has
 \begin{equation}
   \label{eq:calcphi}
   \phi(H)= a - \max\{a_1-a_0,a_2-a_0,a_3,\ldots,a_{10},a_0-a_1-a_2\}.
 \end{equation}
Indeed, for any nontrivial isotropic effective $E \not \equiv
E_{1,2},E_1,\ldots,E_{10}$,
\color{black}
the divisor $E$ intersects  $E_{1,2},E_1,\ldots,E_{10}$ positively by \cite[Lemma~2.1]{klvan}, hence
\color{black}
$E \cdot H \geqslant a \geqslant a-a_i=E_i \cdot H$, for any $i \geqslant 3$. Then \eqref{eq:calcphi} follows since $E_i \cdot H=a+a_0-a_i$ for $i=1,2$
and $E_{1,2} \cdot H=a+a_1+a_2-a_0$. By symmetry, arguing as in the proof of Corollary \ref{cor:kl+},  one can furthermore make sure that 
\begin{equation}
   \label{eq:imp} a_1 >0, \; a_1 \geqslant a_2, \; a_3 \geqslant \cdots \geqslant a_{10} \; \; \mbox{and either} \; \; a_0 >0 \; \mbox{or} \; a_2 \geqslant a_3,
\end{equation}
in which case 
\begin{equation}
   \label{eq:calcphi2}
   \phi(H)= \min\{E_1 \cdot H, E_3 \cdot H, E_{1,2} \cdot H\}=a - \max\{a_1-a_0,a_
3,a_0-a_1-a_2\}.
 \end{equation}
\end{remark}

 We next recall the following from the introduction:

\begin{definition} \label{def:tutte2}
 Two polarized (respectively, numerically polarized) Enriques surfaces $(S,H)$ and $(S',H')$ in $\E_{g,\phi}$ (resp., $(S,[H])$ and $(S, [H'])$ in $\widehat{\E}_{g,\phi}$) {\em  admit the same  simple decomposition type} if there are simple isotropic decompositions
\[H \sim a_1 E_1+\cdots +a_nE_n+\varepsilon K_S
\; \; \mbox{and} \; \; H' \sim a_1 E'_1+\cdots +a_nE'_n+\varepsilon K_{S'}, \; \;  \mbox{with} \; \; \varepsilon  \in \{0,1\}\]
\[ \mbox {(resp.}\;\; H \equiv a_1 E_1+\cdots +a_nE_n
\; \; \mbox{and} \; \; H' \equiv a_1 E'_1+\cdots +a_nE'_n)\]
such that $E_i \cdot E_j=E'_i \cdot E'_j$ for all $i \neq j$.
\end{definition}

\begin{remark} \label{rem:notunique}
  A decomposition type is not necessarily unique within the same linear or numerical equivalence class,   even imposing the conditions \eqref{eq:imp}  on the coefficients. Moreover, also  properties  such as the length or being
$r$-symmetric may  vary with the different ways of writing the
decompositions.  Consider for instance the decomposition type $H
\equiv 2E_1+E_2+E_3+E_4+E_5+E_6+E_{1,7}$ (with $g=30$ and
$\phi(H)=7$),  written in the form of Corollary \ref{cor:kl+}, that is,  $\{E_1,\ldots,E_6\}$  may be extended  to an isotropic $10$-sequence   $\{E_1,\ldots,E_{10}\}$  so that $E_{1,7}$ is defined as in 
\eqref{eq:10-3}. This has length $7$ and is $5$-symmetric, but not
$6$-symmetric.  It therefore does not satisfy the conditions of Theorems \ref{thm:dom} and \ref{thm:dom2}. 
 Let also $E_{7,8}$ be as defined by
\eqref{eq:10-3}. It follows that $E_1+E_{1,7} \sim E_8+E_{7,8}$. Thus,
we may also write $H \equiv E_1+E_2+E_3+E_4+E_5+E_6+E_8+E_{7,8}$,
which has length $8$ and is $6$-symmetric.  This decomposition satisfies the conditions of Theorem \ref{thm:dom2}. 
\end{remark}

\begin{proposition} \label{prop:trans}
The property of admitting the same simple decomposition type defines equivalence relations on $\E_{g,\phi}$ and $\widehat{\E}_{g,\phi}$, respectively.
\end{proposition}

\begin{proof}
The property is clearly reflexive and symmetric, so we have left to prove transitivity. Assume therefore that $(S,H)$ and $(S',H')$ admit the same simple decomposition type and
$(S',H')$ and $(S'',H'')$ admit the same simple decomposition type. We will prove that so do
$(S,H)$ and $(S'',H'')$. The same proof will  work  for numerical polarizations.

 By assumption, using the notation of Corollary \ref{cor:kl+}, we have 
  \begin{eqnarray}
 \label{eq:pippo1}   H  & \sim &  a_0E_{1,2}+a_1E_1+\cdots+a_{10}E_{10} + \varepsilon K_S,  \\
 \label{eq:pippo2}   H'  & \sim &  a_0E'_{1,2}+a_1E'_1+\cdots+a_{10}E'_{10} + \varepsilon K_{S'},  \\
\label{eq:pippo3}  H' & \sim &  
           b_0F'_{1,2}+b_1F'_1+\cdots+b_{10}F'_{10} + \varepsilon' K_{S'}, \\
\label{eq:pippo4}      H''  & \sim &   b_0F''_{1,2}+b_1F''_1+\cdots+b_{10}F''_{10} + \varepsilon' K_{S''}.
  \end{eqnarray}
Here all $a_i$ and $b_i$ are nonnegative integers, $\{E_1,\ldots,E_{10}\}$, $\{E'_1,\ldots,E'_{10}\}$,
 $\{F'_1,\ldots,F'_{10}\}$ and $\{F''_1,\ldots,F''_{10}\}$  are isotropic $10$-sequences, $E_{1,2} \sim \frac{1}{3}(E_1+ \cdots +E_{10})-E_1-E_2$, 
and similarly for $E'_{1,2}$, $F'_{1,2}$ and $F''_{1,2}$. Moreover, by 
Remark \ref{rem:eps} and \eqref{eq:pippo2}-\eqref{eq:pippo3} we see that $\varepsilon=\varepsilon'$; more precisely, combining with \eqref{eq:pippo1} and \eqref{eq:pippo4} we have 
\begin{equation} \label{eq:eps22}
\varepsilon=\varepsilon'= \begin{cases} 0, & \hspace{-0.2cm} \mbox{\small{if $H+K_S$, $H'+K_{S'}$, $H''+K_{S''}$ are not $2$-divisible in the Picard group,}} \\
1, & \hspace{-0.2cm} \mbox{if \small{$H+K_S$, $H'+K_{S'}$, $H''+K_{S''}$ are  $2$-divisible in the Picard group.}} 
\end{cases}
\end{equation}

Now denote by '$[\;]$' the numerical equivalence classes of all divisors above.
Choose isomorphisms $\psi: \Num(S) \stackrel{\cong}{\to} U \+ E_8(-1)$
and $\varphi: \Num(S') \stackrel{\cong}{\to} U \+ E_8(-1)$ (cf. \cite[Lemma VIII.15.1]{BHPV}).
The orthogonal group on $U \+ E_8(-1)$ acts transitively on the set of isotropic $10$-sequences by \cite[Lemma 2.5.2]{cd}, whence we may find an element $\sigma$ of this group such that $\sigma\varphi([E'_i])=\psi([E_i])$ for $1 \leq i \leq 10$. As $[E_{1,2}] = \frac{1}{3}([E_1]+ \cdots +[E_{10}])-[E_1]-[E_2]$ and $[E'_{1,2}] = \frac{1}{3}([E'_1]+ \cdots +[E'_{10}])-[E'_1]-[E'_2]$, we also have $\sigma\varphi([E'_{1,2}])=\psi([E_{1,2}])$. It follows from \eqref{eq:pippo1}-\eqref{eq:pippo2} that $\psi^{-1}\sigma\varphi([H'])=[H]$. By \eqref{eq:pippo3} we also have
\[ \psi^{-1}\sigma\varphi([H'])= 
           b_0\psi^{-1}\sigma\varphi([F'_{1,2}])+b_1\psi^{-1}\sigma\varphi([F'_1])+\cdots+b_{10}\psi^{-1}\sigma\varphi([F'_{10}]).\]
Setting $[E''_{1,2}]:=\psi^{-1}\sigma\varphi([F'_{1,2}])$ and $[E''_i]:=\psi^{-1}\sigma\varphi([F'_i])$, we thus have
\begin{equation} \label{eq:SID}
 [H] = b_0[E''_{1,2}]+b_1[E''_1]+\cdots+b_{10}[E''_{10}],
\end{equation}
where $[E''_{1,2}]  =  \frac{1}{3}\left([E''_1]+ \cdots +[E''_{10}]\right)-\color{black} [E''_1]-[E''_2]\color{black}$ and $\{[E''_1],\ldots,[E''_{10}]\}$ is an isotropic $10$-sequence in $\Num(S)$. We have $[E''_{1,2}] \cdot [H]=2b_1+2b_2+b_3+\cdots+b_{10}=[F'_{1,2}] \cdot [H'] >0$ (as $F'_{1,2}$ is effective) and likewise $[E''_i] \cdot [H] >0$ for $1 \leq i \leq 10$. Hence, by Riemann-Roch and Serre duality, $[E''_{1,2}]$ and $[E''_i]$, $1 \leq i \leq 10$, can be represented by effective classes in $\Pic(S)$.  It follows that \eqref{eq:SID} is a simple isotropic decomposition in $\Num(S)$. By Corollary \ref{cor:kl+} 
and Remark \ref{rem:eps} 
we may find representatives $E''_{1,2}$ and $E''_i$, respectively, so that $E''_{1,2} \sim  \frac{1}{3}(E''_1+ \cdots +E''_{10})-\color{black} E''_1-E''_2 \color{black}$, and
\[
 H \sim b_0E''_{1,2}+b_1E''_1+\cdots+b_{10}E''_{10}+ \varepsilon K_S.
\]
Thus, comparing with \eqref{eq:pippo4}, recalling \eqref{eq:eps22}, we see 
that  $(S,H)$ and $(S'',H'')$ admit the same simple decomposition type. 
\end{proof}

\begin{proposition} \label{prop:desccomp}
  Two numerically polarized Enriques surfaces $(S,[H])$ and  $(S',[H'])$  lie in the same  irreducible  component of $\widehat{\E}_{g,\phi}$ if and only if  they admit the same  simple decomposition type.
\end{proposition}

\begin{proof}
   Since $\Num(S) \cong U \+ E_8(-1)$  is constant among all $S \in \E$, the only if part is immediate.

Conversely, it is proved in \cite{GrHu} that the  irreducible  components of $\widehat{\E}_{g,\phi}$ correspond precisely to the different orbits of the action of the orthogonal group on $U \+ E_8(-1)$. Since this group acts transitively on the set of isotropic $10$-sequences by \cite[Lemma 2.5.2]{cd}, and  $E_{1,2} \equiv \frac{1}{3}(E_1+\cdots+E_{10})-E_1-E_2$,  we see that any two numerical polarizations  admitting the same  simple decomposition type lie in the same  irreducible  component of $\widehat{\E}_{g,\phi}$, as claimed.
\end{proof}

\begin{question}\footnote{This question has subsequently been answered positively in  \cite[Thm.~1.1]{kn-JMPA}.}
 \label{q:desccomp}
  Does Proposition~\ref{prop:desccomp} also hold for polarized
Enriques surfaces? In other words, is it true that $(S,H)$ and $(S,H')$ lie
in the same irreducible component of $\E_{g,\phi}$ if and only if $H$
and $H'$ admit the same simple decomposition type?  (The ``only if''
part follows as in the first lines of the proof of
\ref{prop:desccomp}, as $\Pic(S) \cong U \+ E_8(-1)\+ \ZZ/2\ZZ$ is
also constant among all $S \in \E$.)

Theorems \ref{thm:dom} and \ref{thm:dom2} give a positive answer in the case of simple decomposition types \color{black} that are \color{black} of length $\leqslant 5$ or  $6$--symmetric.  
\end{question}

The following lemma classifies all possible   equivalence classes of  simple  decomposition types with $\phi \leqslant 5$. Note that all decomposition types do exist on any Enriques surface, by Lemma \ref{lemma:ceraprima}(a) and the existence of isotropic $10$-sequences.

\begin{lemma} \label{lemma:decomp}
  Assume $H$ is an effective line bundle on an Enriques surface $S$ such that $H^2=2(g-1) >0$. If $1 \leqslant \phi(H) \leqslant 5$, the line bundle $H$ has one and only one of the following simple isotropic decompositions:

(i) If $\phi(H)=1$, then $H \sim (g-1)E_1+E_2$. 

(ii) If $\phi(H)=2$, then 
\begin{itemize}
\item $H \sim \frac{g-2}{2}E_1 +E_2 +E_3$ if $g$ is even,
\item $H \sim \frac{g-1}{2}E_1 +E_{1,2}$ or $H \equiv \frac{g-1}{2}E_1 +2E_2$ (with $g \geqslant 5$), if $g$ is odd. 
\end{itemize}

(iii) If $\phi(H)=3$, then 
\begin{itemize}
 \item $H \sim \frac{g-3}{3}E_1+E_2+E_{1,2}$ or $H \sim \frac{g-3}{3}E_1+2E_2+E_3$
 (with $g \geqslant 9$) if $g \equiv 0 \; \mbox{mod} \; 3$,
\item $H \sim \frac{g-4}{3}E_1+E_2+E_3+E_4$ or $H \sim \frac{g-1}{3}E_1+3E_2$
 (with $g \geqslant 10$) if $g \equiv 1 \; \mbox{mod} \; 3$,
\item $H \sim \frac{g-2}{3}E_1+E_2+E_{1,3}$ if $g \equiv 2 \; \mbox{mod} \; 3$.
\end{itemize}

(iv) If $\phi(H)=4$, then 
\begin{itemize}
\item 
  \begin{itemize}
  \item [] $H \sim \frac{g-4}{4}E_1+3E_2+E_3$, $ g \geqslant 16$, or 
  \item [] $H \sim \frac{g-4}{4}E_1+E_2+E_3+E_{1,4}$, $ g \geqslant 12$,
  \end{itemize}
if $g \equiv 0 \; \mbox{mod} \; 4$,
\item  \begin{itemize}
\item [] $H \equiv \frac{g-1}{4}E_1+4E_2$, $g \geqslant 17$, or
\item [] $H \equiv \frac{g-1}{4}E_1+2E_{1,2}$, $g \geqslant 9$, or
\item []  $H \equiv \frac{g-5}{4}E_1+2E_{2}+2E_3$, $g \geqslant 13$, or
\item []  $H \sim \frac{g-5}{4}E_1+2E_{2}+E_{1,2}$, $g \geqslant 13$, 
\end{itemize}
if $g \equiv 1 \; \mbox{mod} \; 4$,
\item\begin{itemize}
\item []$H \sim \frac{g-6}{4}E_1+2E_{2}+E_3+E_4$, $g \geqslant 14$, or
\item []$H \sim  \frac{g-2}{4}  E_{1,2}+E_{1}+E_2$, $g \geqslant  10$,
\end{itemize}
if $g \equiv 2 \; \mbox{mod} \; 4$,
\item \begin{itemize}
\item  [] $H \sim \frac{g-3}{4}E_1+2E_2+E_{1,3}$,  $g \geqslant 15$   or
\item [] $H \sim \frac{g-7}{4}E_1+E_{2}+E_{3}+E_4+E_5$, $g \geqslant  11$, 
\end{itemize}
if $g \equiv 3 \; \mbox{mod} \; 4$.
\end{itemize}

(v) If $\phi(H)=5$, then
\begin{itemize}

\item
  \begin{itemize}
  \item []$H \sim \frac{g-5}{5}E_1+E_2+2E_{1,2}$, $ g \geqslant 15$, or 
\item []$H \sim \frac{g-10}{5}E_1+2E_2+E_3+E_4+E_5$, $ g \geqslant 20$, or 
\item []$H \sim \frac{g-5}{5}E_1+4E_2+E_3$, $ g \geqslant 25$
  \end{itemize}
if $g \equiv 0 \; \mbox{mod} \; 5$,

\item
  \begin{itemize}
  \item []$H \sim \frac{g-11}{5}E_1+E_2+E_3+E_4+E_5+E_6$, $ g \geqslant 16$, or 
  \item []$H \sim \frac{g-6}{5}E_1+2E_2+E_3+E_{1,4}$, $ g \geqslant 21$,
 or
\item []$H \sim \frac{g-1}{5}E_1+5E_2$, $ g \geqslant 26$
  \end{itemize}
if $g \equiv 1 \; \mbox{mod} \; 5$,

\item
  \begin{itemize}
  \item []$H \sim \frac{g-7}{5}E_1+E_2+E_3+E_4+E_{1,5}$, $ g \geqslant 17$, or 
  \item []$H \sim \frac{g-7}{5}E_1+3E_2+E_{1,2}$, $ g \geqslant 22$, or
\item []$H \sim \frac{g-7}{5}E_1+3E_2+2E_3$, $ g \geqslant 22$
  \end{itemize}
if $g \equiv 2 \; \mbox{mod} \; 5$,

\item
  \begin{itemize}
  \item []$H \sim \frac{g-3}{5}E_{1}+2E_{1,3}+E_2$, $ g \geqslant 18$, or 
  \item []$H \sim \frac{g-8}{5}E_1+2E_2+E_3+E_{1,2}$, $ g \geqslant 18$, or
\item []$H \sim \frac{g-8}{5}E_1+3E_2+E_3+E_4$, $ g \geqslant 23$
  \end{itemize}
if $g \equiv 3 \; \mbox{mod} \; 5$,

\item
  \begin{itemize}
  \item []$H \sim \frac{g-9}{5}E_{1}+2E_2+2E_3+E_4$, $ g \geqslant 19$, or 
  \item []$H \sim \frac{g-4}{5}E_{1,2}+E_1+E_2+E_3$, $ g \geqslant 19$,  or
\item []$H \sim \frac{g-4}{5}E_1+3E_2+E_{1,3}$, $ g \geqslant 24$
  \end{itemize}
if $g \equiv 4 \; \mbox{mod} \; 5$,
\end{itemize}
\end{lemma}

\begin{proof}
  The proof is tedious but straightforward and similar to  \cite[pf. of Prop. 1.4 in \S 2.2]{KL1},  and we therefore will leave most of it to the reader.  The idea is to pick an effective, isotropic $E$ such that $E \cdot H=\phi(H)$, find a suitable integer $k$ so that $\phi(H-kE) < \phi(H)$ (in which case we use the classification for lower $\phi$), or so that $\phi(H-kE)=\phi(H)$ and $H-kE$ is as in Proposition \ref{prop:bound}(i) or (ii). As a sample, we show how this works in the case $\phi(H)=5$ and $g \equiv 3 \; \mbox{mod} \; 5$. 

We pick an effective, isotropic $E$ such that $E \cdot H=\phi(H)=5$ and set $k:=\frac{g-13}{5}$. Then $(H-kE)^2=24$, so that $\phi(H-kE) \leqslant 4$ by \eqref{eq:bound}.   \medskip

Assume $\phi(H-kE) = 4$ and note that $E\cdot (H-kE)=E\cdot H=5$. By the classification in the case $\phi=4$, we have the three possibilities, where we use Notation \ref{not:int}:
\begin{itemize}
\item[(a)] $H-kE \sim 3F_1+2F_{1,2}$,
 \item[(b)] $H-kE  \equiv   2(F_1+F_2+F_3)$, 
\item[(c)] $H-kE \sim 2F_1+2F_2+F_{1,2}$.
\end{itemize}

Case (b) is impossible, as $5=E \cdot (H-kE)$.  

In case (a) we have $F_1 \cdot (H-kE)=4$ and $F_{1,2} \cdot H=6$, hence $E \not \equiv F_1,F_{1,2}$. Thus, $E \cdot F_1=E \cdot F_{1,2}=1$. Let $F:=F_1+F_{1,2}-E$. Then $F^2=0$, $E \cdot F=2$ and $F_1 \cdot F=1$, so that $F$ is effective, non--zero and we have
\[H \sim kE+3F_1+2F_{1,2} \sim (k+2)E+F_1+2F.\]
Using Notation \ref{not:int}, we set $E_1:=E$, $E_2:=F_1$ and $E_{1,3}:=F$ and, recalling that $k+2=\frac{g-3}{5}$, we obtain the desired form
\begin{equation} \label{eq:form1} 
 H \sim \frac{g-3}{5}E_1+E_2+2E_{1,3}.
\end{equation}
As $5 =\phi(H) \leqslant E_2 \cdot H=\frac{g-3}{5}+2$, we have $g \geqslant 18$. 

In case (c) we have $F_1 \cdot (H-kE)=F_2 \cdot (H-kE)=4$ and $F_{1,2} \cdot (H-kE)=8$, hence $E \not \equiv F_1,F_2,F_{1,2}$. Thus, $E \cdot F_1=E \cdot F_2 = E \cdot F_{1,2}=1$. Let $F:=F_2+F_{1,2}-E$. Then $F^2=0$, $E \cdot F=F_1 \cdot F=2$
and $F_2 \cdot F=1$ and we have
\[ H \sim kE+2F_1+2F_2+F_{1,2} \sim (k+1)E+2F_1+F_2+F.\]
Using Notation \ref{not:int}, we set $E_1:=E$, $E_2:=F_1$, $E_3:=F_2$ and $E_{1,2}:=F$ and, recalling that $k+1=\frac{g-8}{5}$, we obtain the desired form
\begin{equation} \label{eq:form2} 
H \sim \frac{g-8}{5}E_1+2E_2+E_3+E_{1,2}.
\end{equation}
As $5 =\phi(H) \leqslant E_2 \cdot H=\frac{g-8}{5}+3$, we have $g \geqslant 18$.

We claim that $H$ cannot simultaneously have a simple isotropic decomposition as in \eqref{eq:form1} and \eqref{eq:form2}. Indeed,  there are two (respectively, three) isotropic, effective classes $F \in \Num(S)$ such that $F \cdot H=\frac{g+7}{5}$ in case \eqref{eq:form2} if $g >18$ (resp., $g=18$), namely $F \equiv E_2,E_3$ (resp., $F \equiv E_1,E_2,E_3$), whereas there is only one (resp., two) such classes in case \eqref{eq:form1}, namely $F \equiv E_2$ (resp., $F \equiv E_1,E_2$), as  $E_{1,3} \cdot H= \frac{2g-1}{5}>\frac{g+7}{5}$  and $F \cdot H \geqslant \frac{g-3}{5}+1+2=\frac{g+12}{5}$ for $F \not \equiv E_1,E_2,E_{1,3}$ by \cite[Lemma 2.1]{klvan}. \medskip

Assume  $\phi(H-kE)  = 3$.  By the classification in the case $\phi=3$, we have the two possibilities:
\begin{itemize}
\item[(d)] $H-kE \sim 3F_1+F_{2}+F_3+F_4$,
 \item[(e)] $H-kE \sim 4F_1+3F_2$
\end{itemize}
 
In case (d) we have $F_1 \cdot (H-kE)=3$, hence $E \not \equiv F_1$. Thus, we must have $E \cdot F_1=1$ and, possily after rearranging indices, $E \cdot F_2 =E \cdot F_3=1$ and $E \equiv F_4$. Thus, using again Notation \ref{not:int}, we set $E_1:=E$, $E_2:=F_1$, $E_3:=F_2$ and $E_{4}:=F_3$ and, recalling that $k+1=\frac{g-8}{5}$, we obtain the desired form
\begin{equation} \label{eq:form3} 
H \sim \frac{g-8}{5}E_1+3E_2+E_3+E_{4},
\end{equation}
possibly after substituting $E_4$ with $E_4+K_S$. Since  $5 =\phi(H) \leqslant E_2 \cdot H=\frac{g-8}{5}+2$, we obtain $g \geqslant 23$. 
 
 Because of the different values of $\phi(H-kE)$, it is again not possible that $H$ can be written both as in \eqref{eq:form3} and as in \eqref{eq:form1} or \eqref{eq:form2}.  

In case (e) we have $F_1 \cdot (H-kE)=3$ and $F_2 \cdot (H-kE)=4$, whence $E \not \equiv F_1, F_2$. It follows that $E \cdot F_1>0$ and $E \cdot F_2 >0$, so that $5=E \cdot (H-kE) \geqslant 7$, a contradiction.\medskip

Assume  $\phi(H-kE)  = 2$.  By the classification in the case $\phi=2$, we have 
the two possibilities:
\begin{itemize}
\item[(f)] $H-kE \sim 6F_1+F_{1,2}$,
 \item[(g)] $H-kE \equiv 6F_1+2F_2$.
\end{itemize}
In both cases,  since $F_1 \cdot (H-kE)=2$, we have $E \not \equiv F_1$, whence the contradiction $5=E \cdot(H-kE) \geqslant 6 E \cdot F_1 \geqslant 6$.\medskip

Assume finally $\phi(H-kE) = 1$. By the classification in the case $\phi=1$, we have 
$H-kE \sim 12F_1+F_2$. As $F_1 \cdot (H-kE)=1$, we have $E \not \equiv F_1$, whence the contradiction $5=E \cdot(H-kE) \geqslant 12 E \cdot F_1 \geqslant 12$.
\end{proof}

\begin{remark} \label{rem:decomp} We will later use the observation immediately deduced from parts (i)-(ii) of Lemma \ref {lemma:decomp} that for $\phi(H)\leqslant 2$  there are at most three \color{black} distinct \color{black} numerical, effective, isotropic classes $E$ such that $E \cdot H \leqslant 2$.
\end{remark}

 We will now prove 
Proposition \ref{prop:kl+}. First we need three auxiliary results.

\begin{lemma} \label{lemma:ext1}
Let $\{E_1,\ldots,E_r\}$ be an isotropic $r$-sequence with $2
\leqslant r \leqslant 9$, 
and $F$ an isotropic divisor such that $F \cdot E_1=F \cdot E_2=2$ and
$F \cdot E_i=1$ for all $i \in\{3,\ldots, r\}$.
Then there is an isotropic $10$-sequence $\{E_1,\ldots,
E_r,E_{r+1},\ldots, E_{10}\}$ such that $F \cdot E_i=1$ for all $i
\in\{r+1,\ldots,10\}$.
\end{lemma}

\begin{proof}
  The divisor $D:=E_1+E_2+F$ satisfies $D^2=10$ and $\phi(D)=3=E_i \cdot D$ for all $i \in \{1,\ldots,r\}$. Thus, $3D \sim E_1+\cdots+E_{10}$ for an isotropic $10$-sequence $\{E_1,\ldots,  E_{10}\}$ by Lemma \ref{lemma:ceraprima}(b). Since $F \cdot D=4$, we have $F \not \equiv E_i$ for any $i$, hence $F \cdot E_i >0$ for all $i$ by \cite[Lemma 2.1]{klvan}. As 
$12= 3F \cdot D= F \cdot (3D)= 4+F \cdot (E_3+\cdots+E_{10})$, we must have $F \cdot E_i=1$ for all $i$.
\end{proof}

\begin{lemma} \label{lemma:ext0}
  Let $\{E_1,\ldots,E_8,F\}$ be an isotropic $9$-sequence. Then, for any extensison of $\{E_1,\ldots,E_8\}$ to an isotropic $10$-sequence $\{E_1,\ldots,E_{10}\}$, we have either
  \begin{itemize}
  \item[(i)] $F \equiv E_i$, for $i=9$ or $10$, or
  \item[(ii)] $F \cdot E_9=F \cdot E_{10}=2$.
  \end{itemize}
\end{lemma}

\begin{proof}
 If $F \cdot E_i=0$ for $i=9$ or $10$, then $F \equiv E_i$ by \cite[Lemma 2.1]{klvan} and we are done. Otherwise, as
$E_1 +\cdots+E_{10}$ is $3$-divisible by Lemma \ref{lemma:ceraprima}, we must have 
 \[
  F \cdot (E_9+E_{10}) \equiv 1 \; \mod 3 \; \; \mbox{and} \; \; F \cdot E_i>0 \; \;  \mbox{for}\; \;  i=9,10.
\]
We are therefore done if we show that 
\begin{equation}
  \label{eq:mendi3}
  F \cdot E_i \leqslant 2, \; \; \mbox{for} \; \; i \in \{9,10\}.
\end{equation}
To prove this, assume by contradiction that $n:=F \cdot E_9 \geqslant 3$, say. Set
$k= \lfloor \frac{n-1}{2} \rfloor \geqslant 1$ and $B:=F+E_9-kE_1$. Then $B^2 \in \{2,4\}$ and $E_i \cdot B=2-k \leqslant 1$ for all $i \in \{2,\ldots,8\}$, contradicting Remark \ref{rem:decomp}. This proves \eqref{eq:mendi3}, whence the lemma.
\end{proof}

\begin{lemma} \label{lemma:ext2}
  Let $F_1$ and $F_2$ be isotropic divisors such  that  $F_1 \cdot F_2=2$ and $\{E_1,\ldots,E_r\}$ be an isotropic $r$-sequence, with $0 \leqslant r \leqslant 8$, such that $F_i \cdot E_j=1$ for all $i \in \{1,2\}$, $j \in \{1,\ldots,r\}$. 
  
  Then, for $k=1$ or $2$, there is an isotropic $10$-sequence $\{F_k,E_1,\ldots,E_r,E_{r+1},\ldots,E_9\}$ such that, for $j \neq k$, $F_j \cdot E_i=1$ for $i \in\{r+1,\ldots,8\}$ and $F_j \cdot E_9=2$. 
\end{lemma}

\begin{proof}
Assume first that $r \leqslant 7$. 
By  Proposition \ref{prop:cd}, the set $\mathcal A$ of $A \in \Pic(S)$ such that 
\[ A^2=0,  \; A \cdot F_1=A \cdot E_1=\cdots= A \cdot E_r=1,
 \; A \not \equiv E_1+E_2+E_3-F_2 \; \mbox{if} \; r=3,
\]
is nonempty. Pick $A\in \mathcal A$ such that $A \cdot F_2$ is minimal. \medskip

\noindent {\bf Claim.}  $A \cdot F_2 \leqslant 2$. \medskip

Assume, to get a contradiction, that $n:=A \cdot F_2 \geqslant 3$. Let $k=\lfloor \frac{n-1}{3}\rfloor$ and set $B:=A+F_2-kF_1$. Then $2 \leqslant B^2 \leqslant 6$ and $B$ has a simple isotropic decomposition containing at least two summands.
None of these may be $F_2$, since $B-F_2=A-kF_1$ \color{black} has negative square\color{black}, unless $k=0$, in which case $B=F_2+A$ is not a simple isotropic decomposition.  

Since $F_2 \cdot B=n-2k$, the intersection of $F_2$ with each of the summands in the simple
isotropic decomposition of $B$ is smaller than $n$. 
Since $F_1 \cdot B=3$, there is at least one of these summands, say $E'$, such that $F_1 \cdot E'=1$. If $r=0$, since $F_2 \cdot E' <n$, the curve $E'$ contradicts the minimality of $A$ and finishes the proof in this case. 

If $r>0$, then, as $E_i \cdot B =2-k$ for any $i \in \{1,\ldots,r\}$, we must have $ k \leqslant 1$. \medskip

\noindent {\bf Case $k=0$.} Then $n=3$, $B \sim A+F_2$, $B^2=6$ and $\phi(B)=E_i \cdot B=2$. 
Thus, by Lemma \ref {lemma:decomp}\color{black}(ii)\color{black}, $B$ can be written as a sum of three isotropic divisors, containing all $E_i$ for  $i \in \{1,\ldots,r\}$. This implies $r\leqslant 3$. 
Since $F_i \cdot B=3$, for $i=1,2$, each summand has intersection one with $F_i$, for $i=1,2$. This implies $r=3$. Indeed, if $r<3$, then at least one of the summands of $B$, say $E'$, is different from the $E_i$s, and has $E'\cdot E_i=1$ for $i=1,\ldots, r$. Hence $E'\in \mathcal A$ and $E'\cdot F_2=1$, contradicting the minimality of $A$. Since $r=3$, we have $B \equiv E_1+E_2+E_3$. But then $A \equiv E_1+E_2+E_3-F_2$, thus $A\not\in \mathcal A$, a contradiction.  \medskip

\noindent {\bf Case $k=1$.}  One has $B \sim A+F_2-F_1$ and $\phi(B)=E_1 \cdot B=1$. Moreover $B^ 2=2n-6$, hence $(n,B^2) \in \{(4,2),(5,4),(6,6)\}$.\medskip

\noindent {\bf Subcase $(n,B^2) =(4,2)$.} As $E_i \cdot B=1$,  for $i \in \{2,\ldots,r\}$, by Lemma \ref {lemma:decomp}(i) we have $r\leqslant 2$ and, if $r=2$, we have $B\equiv E_1+E_2$. But $3=F_1\cdot B=F_1\cdot (E_1+E_2)=2$, a contradiction. Hence we have $r=1$ and $B \sim E_1+E_2'$ with ${E_2'}^2=0$ and $E_1 \cdot E_2'=1$. 

We have $F_1\cdot B=3$, and since $F_1\cdot E_1=1$, we have $F_1\cdot E'_2=2$. 
Since $F_2\cdot B=2$ and $F_2\cdot E_1=1$, we have  $F_2 \cdot E_2'=1$. 
Set $G:= F_1+F_2+E_2'$. Then $G^2=10$, $F_1 \cdot G=4$ and $\phi(G)=E_1 \cdot G = E_2' \cdot G = F_2 \cdot G=3$.  By Lemma \ref{lemma:ceraprima}(b), we have 
 $3G \sim E_1+E_2' + F_2+F'_1 + \cdots + F'_7$ for an isotropic $10$-sequence $\{E_1,E'_2,F_2,F'_1,\ldots,F'_7\}$.  As $F_1 \cdot (3G)=12$, and $F_1\cdot (E_1+E_2' + F_2)=5$, it follows that 
$F_1 \cdot (F'_1 + \cdots + F'_7)=7$, whence $F_1 \cdot F'_i=1$ for all $i \in\{1,\ldots,7\}$. Since $F_2\cdot F'_i=1$ for  all $i \in\{1,\ldots,7\}$, we find a contradiction to the minimality of $A$. \medskip

\noindent {\bf Subcase $(n,B^2) =(5,4)$.} As $E_i \cdot B=1$,  for $i \in \{2,\ldots,r\}$, by Lemma \ref {lemma:decomp}(i) we have $r=1$ and  $B \sim 2E_1+E_2'$ with ${E_2'}^2=0$ and $E_1 \cdot E_2'=1$.  
As $F_1 \cdot B=F_2 \cdot B=3$, it follows that
$F_1 \cdot E_2'=F_2 \cdot E_2'=1$, contradicting the minimality of $A$.\medskip

\noindent {\bf Subcase $(n,B^2) =(6,6)$.} As $E_1 \cdot B=1$ and $F_1 \cdot B=3$, we must have $B \equiv 3E_1+F_1$. But then we get  the contradiction
\[ 4=F_2 \cdot (A+F_2-F_1)=F_2 \cdot B = 3E_1 \cdot F_2+ F_1 \cdot F_2=5.\]

\medskip

Therefore, we have proved the claim that $A \cdot F_2 \leqslant 2$. \medskip

 Assume now that $A \cdot F_2 =2$.  By Lemma \ref{lemma:ext1}, the isotropic sequence $\{F_1,A,E_1,\ldots,E_r\}$ can be extended to an isotropic  $10$-sequence such that  $F_2 \cdot F_1=F_2 \cdot A\color{black}=2$ \color{black} and $F_2$ has intersection  one with the remaining divisors in the sequence. Hence, we are done.

 Assume next that $A \cdot F_2=1$. We then   repeat the process starting with the isotropic $(r+1)$-sequence $\{E_1,\ldots,E_r,E_{r+1}:=A\}$, unless $r+1=8$. We thus reduce to proving the  lemma when  $r=8$. 

 For the rest of the proof we therefore let $r=8$. Then  we can by  Proposition \ref{prop:cd} 
extend $\{E_1,\ldots, E_8\}$ to an isotropic $10$-sequence $\{E_1,\ldots,E_{10}\}$. We  claim that
\begin{equation}
  \label{eq:almenounozero}
 \mbox{there is an $i \in \{1,2\}$ and a $j \in \{9,10\}$ such that $F_i \equiv E_j$.}
\end{equation}
 Indeed, if not, by Lemma \ref{lemma:ext0} we must have 
all $F_i \cdot E_j=2$ for $i \in \{1,2\}, j \in \{9,10\}$. Set $B:= F_1+F_2+E_9+E_{10}-2E_1$. Then $B^2=6$ and $E_j \cdot B=2$ for all $j \in \{2,\ldots,8\}$, which is impossible by Remark \ref{rem:decomp}. This proves \eqref{eq:almenounozero}. 

By \eqref{eq:almenounozero} we have, say, $F_1 \equiv E_{10}$. Then $E_9 \not \equiv F_2$, so $F_2 \cdot E_9=2$ by Lemma \ref{lemma:ext0}.  Hence, $\{F_1,E_1,\ldots,E_8,E_9\}$ is the desired isotropic $10$-sequence.
\end{proof}

\begin{proof}[Proof of Proposition \ref {prop:kl+}.]  We first prove the first statement.  Consider the simple isotropic  set $\{E_1,\ldots,E_r\}$ satisfying \eqref {eq:int2-int}.  
If  $E_i\cdot E_j=1$  for all $i\neq j$, and if $r\neq 9$, we  apply  Proposition \ref {prop:cd}. If instead $r=9$, we  apply  Lemmas \ref {lemma:ext0} and \ref{lemma:ceraprima}(b).   If  
$E_1\cdot E_2=2$   and otherwise  $E_i\cdot E_j=1$   for $i\neq j$, we  apply  Lemmas \ref {lemma:ext2} and \ref{lemma:ceraprima}(b).   Finally, if  $E_1\cdot E_2=E_1\cdot E_3=2$  and otherwise  $E_i\cdot E_j=1$  for $i\neq j$, we apply  Lemmas \ref {lemma:ext1} and \ref{lemma:ceraprima}(b). 

 To prove the last statement, assume $\I$ satisfies the conditions therein, and that we have extended it to the maximal simple isotropic set $\I'=\{E_{1,2},E_1,\ldots, E_{10}\}$. We must have, possibly after reordering,
$\I=\{E_{1,2},E_1,E_3,\ldots,E_n\}$, for $n \leq 9$. Then $\I$ can also be extended to $\I''=\{E_{1,2},E_1,E_3,\ldots,E_9,E_{1,10},E_{2,10}\}$, which satisfies the condition that the only mutual intersections different from $1$ are $E_1 \cdot E_{1,2}=E_1 \cdot E_{1,10}=2$. 
\end{proof}

\section{Irreducibility,  unirationality and uniruledness of moduli spaces} \label{sec:bas0}

 To prove our results, we  extend a construction from \cite{Ve}. First we recall some basic facts about classical Enriques sextic surfaces in $\PP^ 3$ (see \cite {cd}). 

Fix homogeneous coordinates $(x_0:x_1:x_2:x_3)$ on $\PP^3$ and let
\[T =Z(x_0x_1x_2x_3)\]
 be the  \emph{coordinate tetrahedron}.
 We label by
$\ell_1,\ell_2,\ell_3,\ell'_1,\ell'_2,\ell'_3$ the edges of $T$, in
such a way that $\ell_1,\ell_2,\ell_3$ are coplanar, and $\ell'_i$ is
skew to $\ell_i$ for all $i=1,2,3$.

Consider the linear system $\mathcal S$ of surfaces of degree 6  that  are singular along the edges of $T$. They are called \emph {Enriques sextic surfaces}  and  have equations of the form
\begin{equation} \label{eq:sigma}
c_3(x_0x_1x_2)^2+c_2(x_0x_1x_3)^2+c_1(x_0x_2x_3)^2+c_0(x_1x_2x_3)^2+Q\color{black}(x_0,...,x_3)\cdot\color{black}x_0x_1x_2x_3=0,
\end{equation} 
where $Q\color{black}(x_0,...,x_3)\color{black}=\sum_{i \leqslant j}q_{ij}x_ix_j$ \color{black} and $c_0,\ldots,c_3,q_{ij}$ are constants. \color{black} 
This shows that $\dim(\mathcal S)=13$ and we may identify $\mathcal S$ with the $\PP^ {13}$ with  homogeneous coordinates
\[ q=(c_0:c_1:c_2:c_3:q_{00}:q_{01}:q_{02}:q_{03}:q_{11}:q_{12}:q_{13}:q_{22}:q_{23}:q_{33}). \]

If $\Sigma\in \mathcal S$ is a general surface, its normalization $\varphi:S\to \Sigma$ is an Enriques surface and  $H=\varphi^* (\mathcal O_\Sigma (1))$  is an ample divisor class with $H^ 2=6$ and $\phi(H)=2$. More precisely,  $H \sim E_1+E_2+E_3$, with the usual Notation \ref{not:int},  and the edges $\ell_i$ and $\ell'_i$  of $T$ are the images  by $\varphi$  of the curves $E_i$ and  $E_i' \sim E_i+K_S$, with $i=1,2,3$. (Recall that for a primitive, isotropic $E$, the complete linear system $|E+K_S|$ has a unique  element.)  

We  thus have a natural rational map
\[ p: \mathcal S  \dasharrow \E_{4,2},\]
assigning to a general surface $\Sigma \in \mathcal S$ the pair $(S,H)$, where $\varphi:S\to \Sigma$ is the normalization and  $H=\varphi^* (\mathcal O_\Sigma (1))$.  Composing with the forgetful map $\E_{4,2}\to \E$, we have a rational map
$ \mathcal S  \dasharrow \E$,
which is dominant. Indeed, given  a general, whence unnodal,  Enriques surface $S$, we can find a 3-isotropic sequence $\{E_1,E_2,E_3\}$.  If we set $H=E_1+E_2+E_3$, then $(S,H)\in  \E_{4,2}$ and the linear system $|H|$ determines a morphism $\varphi_H: S\to \PP^ 3$, cf., e.g., \cite[Thm. 4.6.3 and 4.7.2]{cd}, 
and, up to a change of coordinates, $\Sigma=\varphi_H(S)$ is an Enriques sextic surface. Accordingly, the map $p$ is dominant. If $(S,H)$ is a point of  $\E_{4,2}$, the fibre $p^ {-1}(S,H)$ consists of the orbit of $\Sigma=\varphi_H(S)$ via the $3$--dimensional group of projective transformations  fixing   $T$.

 Denote by $v$ the vertex of $T$ not contained
in the face spanned by $\ell_1,\ell_2,\ell_3$.
We define  $\F_i$, $i=0,1,2$,  to be  the family of \color{black} irreducible \color{black} cubic  (resp., quartic, quintic) curves $F \subset \PP^3$ \color{black} of arithmetic genus 1 \color{black} such that $v \not \in F$ and $F$ meets\\ 
\begin{inparaenum}
\item [$\bullet$] all edges of $T$ exactly once, if $i=0$;\\
\item [$\bullet$]  the edges $\ell_1$ and $\ell'_1$ of $T$ exactly twice, and the remaining edges exactly once, if $i=1$;\\
\item [$\bullet$] the edges $\ell_3$ and $\ell_3'$ of $T$ exactly once, and the remaining edges exactly twice, if $i=2$.
\end{inparaenum}

\color{black} Note that if $S$ is an Enriques surface that is
the normalization of a sextic $\Sigma \in \mathcal{S}$ containing an
elliptic curve $F$ as above,
then $\{E_1,E_2,E_3, F\}$ is a simple isotropic set on $S$,
where we still denote by $F$ the strict transform of $F\subset \Sigma$
in $S$. \color{black} In particular, since such simple isotropic sets exist with $F$ irreducible on an unnodal Enriques surface, the families $\F_i$ are non--empty. 
\color{black}

\begin{lemma} \label{lemma:irr+solouna}
(a) The family $\F_2$ is irreducible, $10$-dimensional and rational, and each $F \in \F_2$ is contained in a $3$-dimensional linear system of Enriques sextics.

(b) The family $\F_1$ is irreducible, $8$-dimensional and rational, and each $F \in \F_1$ is contained in a $5$-dimensional linear system of Enriques sextics.

(c) The family $\F_0$ is irreducible, $6$-dimensional and rational, and each $F \in \F_0$ is contained in a $7$-dimensional linear system of Enriques sextics.
\end{lemma}

\begin{proof} We first prove (b) (resp. (c)). Let $F \in \F_1$ (resp. $F \in \F_0$). 
The linear system $\mathcal S$ cuts out on $F$ a linear system of divisors with base locus (containing) $T \cap F$ and a moving part $\mathfrak g$ of degree (at most) $8$ (resp., $6$). Note that $\mathcal S$ contains the $9$--dimensional linear system formed by surfaces of the form $T+Q$, where $Q$ is a general quadric in $\PP^ 3$: looking at equation \eqref {eq:sigma}, these are the surfaces obtained by setting $c_i=0$, for $i=1,\ldots,4$. Since quadrics cut out on $F$ a complete linear system, we see that $\mathfrak g$ is complete, of dimension $7$ (resp. $5$). This proves that the linear system of Enriques sextics containing $F$ has dimension $5$ (resp., $7$). 

\color{black} We now prove the rest of (b). Given $F\in \mathcal F_1$, the intersection of $F$ with the edges of $T$ is a subscheme $Z$ of length 8 of the union of these edges off the vertices of $T$. Let $\mathcal Z$ be the Hilbert scheme of such subschemes. We have a restriction morphism
$\gamma: \mathcal F_1\to \mathcal Z$. \medskip

\noindent {\bf Claim.} The morphism $\gamma$ is injective and dominant. \medskip

Indeed, let $F$ be in $\mathcal F_1$ and let $Z=\gamma(F)$. To prove the injectivity, it suffices to prove that the linear system of quadrics passing through $Z$ has dimension 1. Suppose this is false. \color{black} Then there would be a net $\Q$ of quadrics through these $8$ points.  Fix the attention on a face $\Pi$ of $T$ containing four of these points (on three edges). \color{black} Then the quadrics in $\Q$ containing two fixed general points of $\Pi$ contain $\Pi$, because there is no conic containing the four points of $Z$ on $\Pi$ and two general points of $\Pi$. \color{black} 
Consequently, the remaining four of the eight points should be coplanar, a contradiction\color{black}, proving the injectivity. The dominance is then clear because, consequently, the quadrics containing a general $Z$ in $\mathcal Z$ form a pencil whose base locus is in $\mathcal F_1$. \medskip

Since $\mathcal Z$ is birational to \color{black} $\Sym^2(\ell_1) \x \Sym^2(\ell'_1) \x \ell_2 \x \ell'_2 \x \ell_3 \x \ell'_3 \cong \PP^8$, \color{black} the claim yields that $\mathcal F_1$ is irreducible, rational of dimension 8.  \color{black} This proves (b). 
 
\color{black} We next prove the rest of (c). \color{black}  If $F \in \F_0$, then $F$ spans a plane $\Pi_F \subset \PP^3$, which intersects the set of edges of $T$ in six \color{black} distinct \color{black} points. \color{black} These six points are the vertices of the quadrilateral cut out on $\Pi_F$ by the faces of $T$. Hence the cubic $F$ is smooth at these points, otherwise it would contain one of the sides of the above quadrilateral. Now we claim that \color{black} the set of plane cubics through these six points is a linear system of dimension 3. \color{black} Indeed, otherwise, the cubics in this linear system, of dimension $r\geq 4$, would cut out on $F$, off the six points, a $g^{r-1}_3$ with $r-1\geq 3$, which is impossible, since $F$ has arithmetic genus 1. \color{black} Thus, $\F_0$ is a $\PP^3$-bundle over \color{black} an open subset of \color{black} $|\O_{\PP^3}(1)| \cong \PP^3$, and is therefore irreducible, rational and $6$-dimensional. This proves (c).

As for item (a),  the fact that $\F_2$ is irreducible, $10$-dimensional and rational is proved in \cite[Prop. 1.1 and \S 2]{Ve}. The rest of the assertion is proved exactly in the same way we did it for cases (b) and (c) above.  \end{proof}

We next define $\F_{00}$ to be the family of ordered pairs $(F,F')$ of \color{black} irreducible \color{black} cubic curves $F, F' \subset \PP^3$ \color{black} of arithmetic genus 1 \color{black} such that $F, F' \in \F_0$ and $F$ and $F'$ intersect exactly in one point not on $T$, with distinct tangent lines. 

\color{black} Note that if $S$ is an Enriques surface that is
the normalization of a sextic $\Sigma \in \mathcal{S}$ containing a
pair $(F,F')$ of elliptic curves as above,
then $\{E_1,E_2,E_3, F, F'\}$ is a simple isotropic set on $S$,
where we still denote by $F$ and $F'$ the respective strict transforms
of $F$ and $F'$ in $S$. \color{black} As above, since such  isotropic sets of irreducible curves exist on an unnodal Enriques surface, we have that  $\F_{00}$ is non--empty. 
\color{black}

\begin{lemma} \label{lemma:irr+solouna2}
  The family $\F_{00}$ is irreducible, $11$-dimensional and rational and each pair $(F,F') \in \F_{00}$ is contained in a $2$-dimensional linear system of Enriques sextics.
\end{lemma}

\begin{proof}
  The family  $\F_{00}$ can be constructed in the following way: fix a pair of general planes $\Pi$ and $\Pi'$ in $\PP^3$ intersecting along a line $\ell$, and fix a point $p \in \ell$. Consider in both $\Pi$ and $\Pi'$ the family of cubic curves passing through $p$ and the six intersection points of $\Pi$ and $\Pi'$, respectively, with the edges of $T$; each of these is a two-dimensional linear system. Varying  $\Pi$, $\Pi'$ and $p$ and  taking the two families of cubic curves, we obtain all elements of $\F_{00}$. This description shows the rationality and the dimension. 
  
  Now fix $(F,F') \in \F_{00}$ and  let $\mathcal{S}_{F+F'}$ be the linear system of Enriques sextics containing $F \cup F'$. First we prove that $\dim (\mathcal{S}_{F+F'})\geqslant 2$.  Indeed,  the linear system $\mathcal{S}_F$ of Enriques sextics containing $F$ is 7-dimensional by Lemma \ref{lemma:irr+solouna}(c). It cuts on $F'$ a linear system of divisors with base locus (containing) $T \cap F$ and $p=F \cap F'$ and
a moving part of degree (at most) $5$, hence of dimension at most $4$. Therefore, containing $F'$ imposes at most $5$ conditions on $\mathcal{S}_F$. 

Next we prove that $\dim (\mathcal{S}_{F+F'})\leqslant 2$, which will finish our proof. Consider the pair $F \subset \Pi$ and $F' \subset  \Pi'$ in $\F_{00}$, with the planes they span. Set $\ell=\Pi\cap \Pi'$ and $F \cap \ell=\{a,b,p\}$ and $F' \cap \ell = \{a',b',p\}$. Let $\Sigma \in \mathcal{S}_{F+F'}$ be general. Then $\ell$ intersects $\Sigma$ in
six points, among these are $\{a,b,a',b',p\}$; call $p'$ the sixth point. The surface $\Sigma$ intersects $\Pi$ (resp., $\Pi'$) in a cubic $G$ off $F$ (resp., $G'$ off $F'$), passing through $a'$, $b'$ and $p'$ (resp., $a$, $b$ and $p'$), in addition to the 
six intersection points of $\Pi$ (resp., $\Pi'$) with the edges of $T$. \color{black} Then 
 $G$ (resp. $G'$) is uniquely determined by the condition of passing through the six points $\Xi$ (resp. $\Xi'$) of intersection of $\Pi$ (resp. $\Pi'$) with the edges of $T$ and through $a'$, $b'$, $p'$ (resp., $a$, $b$, $p'$). Let us prove this for $G$ (the proof for $G'$ is identical). Suppose there is a pencil of cubics containing $\Xi$ and $a'$, $b'$, $p'$. Since $a'$, $b'$, $p'$ lie on $\ell$, there is a cubic in the pencil containing $\ell$. The remaining conic component of this cubic should pass through $\Xi$, and this is clearly impossibile. \color{black} Consequently, \color{black} as $\Sigma$ varies in $\mathcal{S}_{F+F'}$, the intersection \color{black} $\Sigma \cap (\Pi \cup \Pi')$  may at most vary with the point $p' \in \ell$. Thus the restriction 
$\mathcal{S}_{\Pi \cup \Pi'}$ of 
$\mathcal{S}_{F+F'}$ to $\Pi \cup \Pi'$ is at most one-dimensional. Consider the restriction map $\mathcal{S}_{F+F'}\dasharrow \mathcal{S}_{\Pi \cup \Pi'}$, which is linear, rational and surjective by assumption. Its indeterminacy locus is the unique surface  $T \cup \Pi \cup \Pi'$. Since $\dim(\mathcal{S}_{\Pi \cup \Pi'})\leqslant 1$, we deduce that $\dim(\mathcal{S}_{F+F'})\leqslant 2$, as desired. \end{proof}

We next define $\F_{0i}$, for $i=1,2$, to be the family of ordered pairs of
\color{black} irreducible \color{black}  curves $(F,F')$ in $\PP^3$ \color{black} of arithmetic genus 1 \color{black} such that $F \in \F_0$, $F' \in \F_1$ and $F$ and $F'$ intersect exactly in $i$ points not on $T$, with distinct tangent lines.

\color{black} Note that, as before, if $S$ is an Enriques surface that is
the normalization of a sextic $\Sigma \in \mathcal{S}$ containing a
pair $(F,F')$ of elliptic curves as above,
then $\{E_1,E_2,E_3, F, F'\}$ is a simple isotropic set on $S$,
where we still denote by $F$ and $F'$ the respective strict transforms
of $F$ and $F'$ in $S$.  \color{black} Since such  isotropic sets of irreducible curves exist on an unnodal Enriques surface, we have that  each $\F_{0i}$ is non--empty. 
\color{black}

\begin{lemma} \label{lemma:irr+solouna3}
  The family $\F_{0i}$ is irreducible, uniruled and $(14-i)$-dimensional and each pair $(F,F') \in \F_{0i}$ is contained in a linear system $\mathcal{S}_{F+F'}$ of Enriques sextics of dimension at least $i-1$. If $(F,F') \in \F_{0i}$ is contained in an Enriques sextic $\Sigma$ whose normalization $S$ is an Enriques surface, then $\mathcal{S}_{F+F'}$ has dimension exactly $i-1$, unless $F+F'$  is contained in only {\em nodal} Enriques sextics (that is, Enriques sextics whose normalizations \color{black} are nodal\color{black}). 
\end{lemma}

\begin{proof}
  We have a natural dominant map $q:\F_{0i} \to \F_1 \x (\PP^3)^{\vee}$ sending the pair $(F,F')$ to $F' \in \F_1$ and the plane $\Pi_F$ spanned by $F$ in $(\PP^3)^{\vee}$. 

For $i=1$, the fiber of $q$ over $(F',\Pi)$  consists of the union of four $2$-dimensional linear systems of cubics in $\Pi$ through the six intersection points of $\Pi$ with the edges of $T$ and one of the four intersection points of $\Pi$ with $F'$. This proves the irreduciblity because  the monodromy action of the four intersection points is the symmetric group \color{black}(see \cite[Lemma on p. 111]{ACGH}), \color{black} and shows also the uniruledness. The dimension also follows easily.

For $i=2$, the fiber of $q$ over $(F',\Pi)$  consists of the union of six $1$-dimensional linear systems of cubics in $\Pi$ through the six intersection points of $\Pi$ with the edges of $T$ and two of the four intersection points of $\Pi$ with $F'$. As above, this  proves  irreduciblity, uniruledness and the dimension.

The dimension of the linear system of Enriques sextics $\mathcal{S}_{F'}$ containing a fixed $F' \in \F_1$ is $5$ by Lemma \ref{lemma:irr+solouna}(b). Containing an additional cubic $F \in \F_0$ intersecting $F'$ in $i$ points, imposes at most $6-i$ conditions, arguing as in the proof of Lemma \ref{lemma:irr+solouna2}. Therefore, the linear system of Enriques sextics $\mathcal{S}_{F+F'}$ containing a pair $(F,F') \in \F_{0i}$ has dimension at least $5-(6-i)=i-1$.

Let $\Sigma$ be an Enriques sextic containing $F +F'$  such that its normalization $\varphi: S\to \Sigma$ is an  unnodal  Enriques surface.  
The linear system $\mathcal{S}$ cuts on $\Sigma$ a linear system whose pull--back on $S$ via $\varphi$  is the sublinear system of $|6(E_1+E_2+E_3)|$ with base locus twice the sum of the pullback of the edges of the tetrahedron, which is 
\[2\Big(E_1+E_2+E_3+(E_1+K_S)+(E_2+K_S)+(E_3+K_S)\Big) \sim 4(E_1+E_2+E_3).\]
Hence, the free part is
$|2(E_1+E_2+E_3)|$. So we have a  linear, rational restriction  map 
\[
\mathcal{S}_{F+F'}\dasharrow   |B|, \; \; \mbox{with} \; \; B:=2(E_1+E_2+E_3)-(F+F')
\]
 whose indeterminacy locus is just the surface $\Sigma$.

We have  $B^2=2(i-2)$. If  $i=1$ and $S$ is  unnodal, then  $|B|=\emptyset$,   which shows that $\mathcal{S}_{F+F'}=\{\Sigma\}$ has dimension 0, as wanted.  If $i=2$, then $B^2=0$ and $E_1 \cdot B=1$, hence $h^0(B)=1$ by
Riemann--Roch.   This yields
$\dim(\mathcal{S}_{F+F'})\leqslant 1$, proving the assertion. 
\end{proof}

Consider now the incidence varieties
\[ \G_i :=\{ (F,\Sigma) \in \F_i \x \SS \; | \; F \subset \Sigma\},\]
for $i=0,1,2$, 
and
\[ \G_{00} :=\{ (F,F',\Sigma) \in \F_{00} \x \SS \; | \; F+F' \subset \Sigma\},\]
which are irreducible, rational and $13$-dimensional, by Lemmas \ref{lemma:irr+solouna} and \ref{lemma:irr+solouna2}. Similarly, for $i=1,2$, let
\[ \G_{0i} :=\{ (F,F',\Sigma) \in \F_{0i} \x \SS \; | \; \Sigma \; \mbox{is unnodal}, \; F+F' \subset \Sigma\},\]
which are irreducible, uniruled and $13$-dimensional, by Lemma \ref{lemma:irr+solouna3}.

\begin{prop}\label{prop:domin} If $\G$ is any of the incidence varieties $\G_i$, for $i=0,1,2$, $\G_{00}$, $ \G_{0i}$, for $i=1,2$, the obvious projection  $\pi: \G\to \SS$ is dominant, hence generically finite. Accordingly, if $\xi\in \G$ is a general point, then $\Sigma=\pi(\xi)$ is a general element of $\SS$ and its normalization $S$ is a general Enriques surface.
\end{prop}

\begin{proof} We prove the assertion for $\G=\G_{00}$, the proof in the other cases being similar.

Let $S$ be a general Enriques surface. There is  an isotropic $5$--sequence $\{E_1,\ldots,E_5\}$ on $S$.  Set $H=E_1+E_2+E_3$. Then $\varphi_H: S\to \Sigma\subset \PP^ 3$ maps $S$, up to a projective transformation, to a general surface in 
$\mathcal S$. Moreover $E_4$, $E_5$ are mapped to two elliptic cubic curves $F,F'$ meeting at a point. This proves the assertion. \end{proof}

We now define various maps from these incidence varieties to some $\E_{g,\phi}$s, for various $g$ and $\phi$, which we eventually prove to be dominant, establishing irreducibility and unirationality or uniruledness. 

Consider a general element $(F,\Sigma)$ of $\G_i$, for $i=0,1,2$. Then the normalization $S$ of  $\Sigma$ is an Enriques surface and on $S$ we have the three curves $E_1,E_2,E_3$, plus the strict transform of $F$ which, by abuse of notation, we still denote by $F$. Similar convention we introduce for  $\G_{0i}$, for  $i=0,1,2$.

Fix four nonnegative integers $\alpha_1,\alpha_2,\alpha_3,\alpha_4$, at least two nonzero.  Then, for each $i=0,1,2$,  and $\varepsilon=0,1$,  we have a rational map
\[ f^i_{\alpha_1,\alpha_2,\alpha_3, \alpha_4; \varepsilon}: \G_i \dasharrow \E_{g,\phi}\]
sending the general point $(F, \Sigma)\in \G_i$
to $\left(S,\O_{S}(\alpha_1E_1+\alpha_2E_2+\alpha_3E_3+\alpha_4F+\varepsilon K_S)\right)$, where $g=p_a(\alpha_1E_1+\alpha_2E_2+\alpha_3E_3+\alpha_4F)$  and $\phi=\phi(\alpha_1E_1+\alpha_2E_2+\alpha_3E_3+\alpha_4F)$.

Next, fix five positive integers $\alpha_1,\ldots,\alpha_5$, at least two nonzero. For each $i=0,1,2$,  and $\varepsilon=0,1$,  we have a rational map
\[ f^{0i}_{\alpha_1,\alpha_2,\alpha_3, \alpha_4,\alpha_5; \varepsilon}: \G_{0i} \dasharrow \E_{g,\phi}\]
sending   a general  $(F,F', \Sigma)\in \G_{0i}$  to $\left(S,\O_{S}(\alpha_1E_1+\alpha_2E_2+\alpha_3E_3+\alpha_4F+\alpha_5F'+\varepsilon K_S)\right)$, where
$g=p_a(\alpha_1E_1+\alpha_2E_2+\alpha_3E_3+\alpha_4F+\alpha_5F')$ and $\phi:=\phi(\alpha_1E_1+\alpha_2E_2+\alpha_3E_3+\alpha_4F+\alpha_5F')$.

Let now $(F,\Sigma) \in \G_2$  be general  and consider  the curves $E_1,E_2,E_3, F $  on $S$. Then $E_1+E_2+F$ satisfies the conditions of Lemma \ref{lemma:ceraprima}(b). Since $E_i \cdot (E_1+E_2+F)=3$, for $i=1,2,3$, we obtain an isotropic $10$-sequence $\{E_1,E_2,E_3,E_4,\ldots, E_{10}\}$ such that 
\[ 3(E_1+E_2+F) \sim E_1 + \cdots + E_{10}.\]
\color{black} In particular, $F=E_{1,2}$. \color{black}
Note that each $E_i$ for $i \geqslant 4$ is uniquely determined up to numerical equivalence class and permutation of indices; in particular, $E_4+\cdots+E_{10} \sim
2E_1+2E_2+3F-E_3$ is a well-defined element of $\Pic (S)$. 
For any 
five nonnegative integers $\alpha_0,\ldots,\alpha_4$ such that at least one among $\alpha_0,\ldots,\alpha_3$ is zero, we can consider the rational map
\[ h_{\alpha_0,\alpha_1,\alpha_2,\alpha_3, \alpha_4; \varepsilon}: \G_2 \dasharrow \E_{g,\phi}\]
sending $(F,\Sigma)\in \G_2$ to $(S,\O_{S}(\alpha_0F+\alpha_1E_1+\alpha_2E_2+\alpha_3E_3+\alpha_4(E_4+\cdots+E_{10})+\varepsilon K_S))$, where
$g:=p_a(\alpha_0F+\alpha_1E_1+\alpha_2E_2+\alpha_3E_3+\alpha_4(E_4+\cdots+E_{10}))$ and 
 $\phi:=\phi(\alpha_0F+\alpha_1E_1+\alpha_2E_2+\alpha_3E_3+\alpha_4(E_4+\cdots+E_{10}))$.

Finally, let $(F,F',\Sigma) \in \G_{02}$ be a general point and consider $E_1,E_2,E_3, F,F'$ curves in $S$. Then $F+F'+E_1$ satisfies the conditions of Lemma \ref{lemma:ceraprima}(b). Since $F\cdot (E_1+F+F')=
E_i \cdot (E_1+F+F')=3$, for $i=1,2,3$, we obtain an isotropic $10$-sequence $\{E_1,E_2,E_3,E_4:=F,E_5,\ldots, E_{10}\}$ such that 
\[ 3(E_1+F+F') \sim E_1 + \cdots + E_{10}.\]
\color{black} In particular, $F'=E_{1,4}.$ \color{black}
Note that each $E_i$ for $i \geqslant 5$ is uniquely determined up to numerical equivalence class and permutation of indices; in particular, $E_5+\cdots+E_{10} \sim
2E_1+2F+3F'-E_2-E_3$ is a well-defined element of $\Pic (S)$. For any 
six nonnegative integers $\alpha_0,\ldots,\alpha_5$ such that at least one among $\alpha_0,\ldots,\alpha_4$ is zero, we have a map 
\[ h^0_{\alpha_0,\alpha_1,\alpha_2,\alpha_3, \alpha_4,\alpha_5; \varepsilon}: \G_{02} \dasharrow \E_{g,\phi}\]
sending $(F,F',\Sigma)$ to 
\begin{small}
$(S,\O_{S}(\alpha_0F'+\alpha_1E_1+\alpha_2E_2+\alpha_3E_3+\alpha_4F+\alpha_5(E_5+\cdots+E_{10})+\varepsilon K_S)),$
\end{small} where
$g:=p_a(\alpha_0F'+\alpha_1E_1+\alpha_2E_2+\alpha_3E_3+\alpha_4E_4+\alpha_5(E_5+ \cdots+E_{10}))$ and 
 $\phi:=\phi(\alpha_0F'+\alpha_1E_1+\alpha_2E_2+\alpha_3E_3+\alpha_4E_4+\alpha_5(E_5+\cdots+E_{10}))$.

Our main results, Theorem \ref{thm:dom} and Theorem \ref{thm:dom2}, are, respectively, immediate consequences of the following two propositions \color{black} and the fact that the varieties  $\mathcal G_i$ and $\mathcal G_{0,i}$ are irreducible and unirational or uniruled, as mentioned above. \color{black}  Let as usual $\varepsilon\in\{0,1\}$. 

\begin{proposition} \label{prop:dom} 
Let $i \in \{0,1,2\}$ and $\alpha_1,\ldots,\alpha_4\in \mathbb{N}$, at
least two nonzero.
  The map $f^i_{\alpha_1,\alpha_2,\alpha_3, \alpha_4; \varepsilon }$ 
is dominant onto the locus of pairs $(S,H) \in \E_{g,\phi}$ admitting
the same simple decomposition type as 
$\alpha_1E_1+\alpha_2E_2+\alpha_3E_3+\alpha_4F+\varepsilon K_S$.

Let $i \in \{0,1,2\}$ and $\alpha_1,\ldots,\alpha_5\in \mathbb{N}$, at
least two nonzero.
The map $f^{0i}_{\alpha_1,\alpha_2,\alpha_3, \alpha_4,\alpha_5; \varepsilon }$ 
is dominant onto the locus of pairs $(S,H) \in \E_{g,\phi}$  admitting the same  simple decomposition type as $\alpha_1E_1+\alpha_2E_2+\alpha_3E_3+\alpha_4F+\alpha_5F'+\varepsilon K_S$.
\end{proposition}
 
\begin{proposition} \label{prop:dom2} 
Let $\alpha_0,\ldots, \alpha_4 \in \mathbb{N}$, with at least one among 
$\alpha_0,\ldots,\alpha_3$  being  nonzero.
  The map $h_{\alpha_0,\alpha_1,\alpha_2,\alpha_3, \alpha_4; \varepsilon }$
is dominant onto the locus of pairs $(S,H) \in \E_{g,\phi}$ admitting
the same simple decomposition type as
$\alpha_0F+\alpha_1E_1+\alpha_2E_2+\alpha_3E_3+\alpha_4(E_4+\cdots+E_{10})+\varepsilon K_S$.

Let $\alpha_0,\ldots, \alpha_5 \in \mathbb{N}$, with at least one among 
$\alpha_0,\ldots,\alpha_4$  being  nonzero.
The map $h^{0}_{\alpha_0,\alpha_1,\alpha_2,\alpha_3, \alpha_4,\alpha_5; \varepsilon }$ 
is dominant onto the locus of pairs $(S,H) \in \E_{g,\phi}$  admitting the same  simple decomposition type as $\alpha_0F'+\alpha_1E_1+\alpha_2E_2+\alpha_3E_3+\alpha_4F+\alpha_5(E_5+\cdots+E_{10})+\varepsilon K_S$.
\end{proposition}

The proofs of Propositions~\ref{prop:dom} and \ref{prop:dom2} require
the results of Section~\ref{sec:bas1}  (more precisely, Corollary \ref{cor:kl+})  to make sure we have enough
isotropic divisors in the decompositions of $H$ to map $S$ to an
Enriques sextic in the appropriate way.
For instance, if
$H  \sim  \alpha_1 E_1+\alpha_2 E_{1,2}$, one 
writes $H  \sim  \alpha_1 E_1+\alpha_2 E_{1,2}+0E_2+0E_3$ so that
$E_1+E_2+E_3$ defines a mapping of $S$ to an Enriques sextic
(following Notation~\ref{not:int} everywhere).

We use the following definition in the proofs of
Propositions~\ref{prop:dom} and \ref{prop:dom2}. 

\begin{definition} \label{def:triple}
 For an isotropic $3$-sequence $\mathfrak{I}=\{E_1,E_2,E_3\}$ on the Enriques surface $S$, we let $\widetilde{\F}_i(\mathfrak{I})$, $i=0,1,2$,   be  the set of all primitive, isotropic divisors $F$ on $S$ satisfying
\[
(F \cdot E_1, F \cdot E_2, F \cdot E_3)= 
\begin{cases} (1,1,1) & \mbox{if} \; \; F \in \widetilde{\F}_0 (\mathfrak{I}), \\
              (2,1,1) & \mbox{if} \; \; F \in \widetilde{\F}_1 (\mathfrak{I}), \\
(2,2,1) & \mbox{if} \; \; F \in \widetilde{\F}_2 (\mathfrak{I})
\end{cases}
\]
and  $\widetilde{\F}_{0i} (\mathfrak{I})$, $i=0,1,2$,  the set of all pairs $(F,F')$ of primitive, isotropic divisors $F, F'$ on $S$ such that
$F \in \widetilde{\F}_0 (\mathfrak{I})$ and
\begin{itemize}
\item $F' \in \widetilde{\F}_0 (\mathfrak{I})$ and $F \cdot F'=1$, if $i=0$,
\item $F' \in \widetilde{\F}_1 (\mathfrak{I})$ and $F \cdot F'=1$, if $i=1$,
\item $F' \in \widetilde{\F}_1 (\mathfrak{I})$ and $F \cdot F'=2$, if $i=2$.
\end{itemize}
\end{definition}

\begin{proof}[Proof of Proposition \ref{prop:dom}]
  Let $(S,H)$ be as in either of the statements of the proposition. In
particular, $H$ admits a simple decomposition type of length $n$, with
$2 \leqslant n \leqslant 5$.  By   Corollary \ref{cor:kl+}, if $n
\leqslant 4$, we may write $H  \sim  
\alpha_1E_1+\alpha_2E_2+\alpha_3E_3+\alpha_4F+\varepsilon K_S$ with $\mathfrak
I=\{E_1, E_2, E_3\}$ an isotropic $3$-sequence and $F \in
\widetilde{\F}_i(\mathfrak I)$,
possibly allowing some of the $\alpha_i$s to be $0$.  If $n=5$, we may
write $H  \sim  
\alpha_1E_1+\alpha_2E_2+\alpha_3E_3+\alpha_4F+\alpha_5F'+\varepsilon K_S$ with $(F,F')
\in \widetilde{\F}_{0i}(\mathfrak I)$.
We may assume $(S,H)$ to be general, in particular, $S$ is
unnodal. Then by \cite[Thm. 4.6.3 and 4.7.2]{cd} the complete linear
system $|E_1+E_2+E_3|$ maps $S$ birationally onto an Enriques sextic
in $\PP^3$, with double lines along the edges of the tetrahedron $T$
defined by the images of all $E_i$ and $E'_i:=E_i+K_S$. Under this
map, $F$ (respectively, $(F,F')$) is mapped to an element of $\F_i$
(resp., $\F_{0i}$), finishing the proof.
\end{proof}

\begin{proof}[Proof of Proposition \ref{prop:dom2}]
 To prove the surjectivity of  \color{black} $h_{\alpha_0,\ldots,\alpha_4;\varepsilon}$\color{black}, assume  
$(S,H)$  admits  the given simple decomposition type as in the statement. We may assume that $\alpha_4>0$, otherwise the result follows from Proposition \ref{prop:dom}. By   Corollary \ref{cor:kl+},   we may always write $H \sim \alpha_0E_{1,2}+\alpha_1E_1+\alpha_2E_2+\alpha_3E_3+\alpha_4(E_4+\cdots+E_{10})+\varepsilon K_S$, possibly allowing more than one among $\alpha_0,\alpha_1,\alpha_2,\alpha_3$ to be zero. Since $E_{1,2} \in \widetilde{\F}_2(E_1,E_2,E_3)$, the result follows as in the proof of Proposition \ref{prop:dom}. 

To prove the surjectivity of $h^0_{\alpha_0,\ldots,\alpha_5;\varepsilon}$, assume 
$(S,H)$  admits  the given simple decomposition type as in the statement. We may again assume that $\alpha_5>0$. By   Corollary \ref{cor:kl+},   we may always write $H \sim \alpha_0E_{1,4}+\alpha_1E_1+\alpha_2E_2+\alpha_3E_3+\alpha_4E_4+\alpha_5(E_5+\cdots+E_{10})$, possibly allowing more than one among $\alpha_0,\alpha_1,\alpha_2,\alpha_3,\alpha_4$ to be zero. Then $(E_4,E_{1,4}) \in \widetilde{\F}_{02}(E_1,E_2,E_3)$ and the result follows as in the proof of Proposition \ref{prop:dom}. 
\end{proof}

We now give the proofs of the three corollaries in the introduction.

\begin{proof}[Proof of Corollary \ref{cor:main1}]
  By Lemma \ref{lemma:decomp}, all cases with $\phi \leqslant 4$  admit  simple decomposition types of length $n\leqslant 4$, except for the decomposition type $\frac{g-7}{4}E_1+E_{2}+E_{3}+E_4+E_5$. The result thus follows from Theorem \ref{thm:dom}. 
\end{proof}

\begin{proof}[Proof of Corollary \ref{cor:main2}]
  Since $g \leqslant 20$, we have $H^2 \leqslant 38$, whence $\phi \leqslant 6$ by \eqref{eq:bound}, with equality $\phi=6$ possible only for $H^2=36$  by Proposition \ref{prop:bound}, in which case the simple decomposition type has  length $2$.    Thus the result follows from Theorem \ref{thm:dom} in this case. 

We have left to treat the cases where $\phi \leqslant 5$. By Lemma \ref{lemma:decomp}, all cases with $\phi \leqslant 5$ and $g \leqslant 20$ have decomposition types of length $n\leqslant 5$, except for the type $E_1+E_{2}+E_{3}+E_4+E_5+E_6$ for $(g,\phi)=(16,5)$, which is the only type occurring for these values of $g$ and $\phi$. Hence $\E_{16,5}$ is irreducible and uniruled by Theorem \ref{thm:dom2}. Again by Lemma \ref{lemma:decomp}, all remaining cases with $\phi \leqslant 5$ and $g \leqslant 20$  admit  simple decomposition types of length $n\leqslant 4$ or of length $5$ with all nonzero intersections occurring equal to one, except for the type $2E_1+E_2+E_3+E_4+E_{1,5}$ for $(g,\phi)=(17,5)$, which is the only type occurring for these values of $g$ and $\phi$. Hence $\E_{17,5}$ is irreducible and uniruled 
and all  irreducible  components of the remaining $\E_{g,\phi}$ are unirational 
by Theorem \ref{thm:dom}.
 \end{proof}

\begin{proof}[Proof of Corollary \ref{cor:conn}]
  If $[H] \in \Num(S)$ is  not $2$-divisible, then  by Lemma \ref{lemma:H2div} 
some simple decomposition types of $H$ and $H+K_S$ have not all even coefficients in front of the isotropic, primitive summands. Hence, by substituting one $E_i$ with odd coefficient with $E_i+K_S$, we see that $H$ and $H+K_S$  admit the same  simple decomposition type, and thus belong to the same  irreducible  component of $\E_{g,\phi}$, by Theorems \ref{thm:dom} or \ref{thm:dom2}  
and the  assumption on the decomposition types. Hence $\rho^{-1}(\rho(\C))$ is irreducible.

Conversely, assume $[H] \in \Num(S)$ is  $2$-divisible.  Then $H$ and $H+K_S$  do not  lie in the same  irreducible  component of $\E_{g,\phi}$  by  
the last assertion in Lemma 
\ref{lemma:H2div},  whence $\rho^{-1}(\rho(\C))$ consists of two disjoint components.  \end{proof}

 Finally, we note that 
our results can also be used to describe the
irreducible components of $\E_{g,\phi}$ for the highest values of
$\phi$ with respect to $g$. Indeed, one has $\phi^2 \leqslant 2(g-1)$
(cf. \cite[Cor. 2.7.1]{cd}) and there are no cases with $\phi^2 <
2(g-1) < \phi^2 +\phi-2$ (cf. \cite[Prop. 1.4]{KL1}). In the bordeline
cases, we obtain:

\begin{corollary} \label{cor:main3}
  For each even $\phi$, the space $\E_{\frac{\phi^2}{2}+1,\phi}$ 
is irreducible and unirational if $ \phi \equiv 2 \; \mbox{mod} \; 4$ and
has two  irreducible  components, both unirational, if $ \phi \equiv 0 \; \mbox{mod} \; 4$.

For each $\phi \geqslant 1$, the space $\E_{\frac{\phi(\phi+1)}{2},\phi}$ is irreducible and unirational when $\phi \neq 6$, and consists of three  irreducible  unirational components when $\phi=6$. 
\end{corollary}

\begin{proof}
  When $g=\frac{\phi^2}{2}+1$, equivalently $H^2=\phi^2$, then Proposition \ref{prop:bound}, Theorem \ref{thm:dom} and Lemma 
\ref{lemma:H2div}  yield that, when $\frac{\phi}{2}$ is even, i.e., $\phi\equiv 0\;\; \mbox{mod} \; 4$ (respectively, when $\frac{\phi}{2}$ is odd, i.e., $\phi\equiv 2\;\; \mbox{mod} \; 4$) then $\E_{\frac{\phi^2}{2}+1,\phi}$ has two irreducible, unirational components (resp. only one irreducible, unirational component), corresponding to the simple decomposition types $\frac{\phi}{2}\left(E_1+E_{1,2}\right)$ and 
$\frac{\phi}{2}\left(E_1+E_{1,2}\right)+K_S$ (resp. $\frac{\phi}{2}\left(E_1+E_{1,2}\right)$). 

When $g=\frac{\phi(\phi+1)}{2}$, Proposition \ref{prop:bound} yields that there is a unique simple decomposition type, of length $3$, 
for each $\phi$, except  for $\phi=6$, where there are three possible types
\[ 2E_1+3E_{1,2}+E_2, \; \; 2(E_1+E_2+E_{1,2}), \; \; 2(E_1+E_2+E_{1,2})+K_S,\]
The result follows from Theorem  \ref{thm:dom}.    
\end{proof}

The cases of the latter corollary are of particular interest from a Brill-Noether theoretical point of view, since they are precisely the cases where the gonality of a general curve in the complete linear system $|H|$ is less than both $2\phi$ and $\lfloor \frac{g+3}{2} \rfloor$, the first being the lowest degree of the restriction of an elliptic pencil on the surface, the latter being the gonality of a general curve of genus $g$, cf. \cite[Cor. 1.5]{KL1}.

\newpage

\section*{Appendix:
  Irreducible components  of  $\widehat{\E}_{g,\phi}$ and  $\E_{g,\phi}$ for  $g \leqslant 30$}

{\small
Using  Proposition \ref{prop:desccomp}  (and Notation \ref{not:int})  we list all  irreducible  components of the moduli spaces $\widehat{\E}_{g,\phi}$ for $g \leqslant 30$,
and describe   the properties of $\rho^{-1}$ of these components obtained by Theorems \ref{thm:dom} and \ref{thm:dom2} and Corollary \ref{cor:conn}. We thus obtain  information about all  irreducible  components of the moduli spaces $\E_{g,\phi}$, with few exceptions\footnote{ The few cases marked with ``??'' in the tables are now known to be irreducible as a consequence of \cite[Thm. 1.1]{kn-JMPA}. Their unirationality/uniruledness, or even Kodaira dimension, is however still open.}. 
The various 
decomposition types can be obtained from 
Lemma \ref{lemma:decomp} and Proposition \ref{prop:bound}, and an ad hoc treatment as in the proof of Lemma \ref{lemma:decomp} for the cases $\phi=6$ and $7$.
The fact that all 
decomposition types below are  in different equivalence classes  can be checked by computing suitable intersections as in the proof of Lemma \ref{lemma:decomp}, and the fact that they all do exist on any Enriques surface follows from Lemma \ref{lemma:ceraprima}(a). 
}

\footnotesize

\noindent
\begin{minipage}[t]{\textwidth/2-1mm}
\vspace{0mm}
\resizebox{\textwidth-4mm}{!}{%
\begin{tabular}{ |l | l |l|l|l|l|}
\hline			
$g$ &  $\phi$  & comp.  & dec. type &  $\rho^{-1}$    
\\\hline\hline

$2$ & $1$ & $\widehat{\E}_{2,1}$ & $E_1+E_2$ &
irred. unirat. \\\hline\hline

$3$ & $1$ & $\widehat{\E}_{3,1}$ & $2E_1+E_2$ & irred. unirat. \\\hline
$3$ & $2$ & $\widehat{\E}_{3,2}$ &  $E_1+E_{1,2}$  & irred. rational \cite{Ca} \\\hline\hline

$4$ & $1$ & $\widehat{\E}_{4,1}$ & $3E_1+E_2$ & irred. unirat. \\\hline
$4$ & $2$ & $\widehat{\E}_{4,2}$ & $E_1+E_2+E_3$  & irred. rational, 
\cite[\S 4]{dol} \\\hline\hline

$5$ & $1$ & $\widehat{\E}_{5,1}$ & $4E_1+E_2$ & irred. unirat. \\\hline
$5$ & $2$ & $\widehat{\E}_{5,2}^{(I)}$ &  $2E_1+E_{1,2}$  & irred. unirat. \\
$5$ & $2$ & $\widehat{\E}_{5,2}^{(II)}$ &  $2(E_1+E_2)$  & two
unirat. components\\\hline\hline 

$6$ & $1$ & $\widehat{\E}_{6,1}$ & $5E_1+E_2$ & irred. unirat. \\\hline
$6$ & $2$ & $\widehat{\E}_{6,2}$ & $2E_1+E_2+E_3$  & irred. unirat. \\ \hline
$6$ & $3$ & $\widehat{\E}_{6,3}$ & $E_1+E_2+E_{1,2}$  & irred. unirat. \cite{Ve} \\ \hline\hline 

$7$ & $1$ & $\widehat{\E}_{7,1}$ & $6E_1+E_2$ & irred. unirat. \\\hline
$7$ & $2$ & $\widehat{\E}_{7,2}^{(I)}$ &  $3E_1+E_{1,2}$  & irred. unirat. \\
$7$ & $2$ & $\widehat{\E}_{7,2}^{(II)}$ &  $3E_1+2E_2$   & irred. unirat. \\\hline
$7$ & $3$ & $\widehat{\E}_{7,3}$ &  $E_1+E_2+E_3+E_4$  & irred. unirat. \\\hline\hline

$8$ & $1$ & $\widehat{\E}_{8,1}$ & $7E_1+E_2$ & irred. unirat. \\\hline
$8$ & $2$ & $\widehat{\E}_{8,2}$ &  $3E_1+E_2+E_3$  & irred. unirat. \\\hline
$8$ & $3$ & $\widehat{\E}_{8,3}$ &  $2E_1+E_2+E_{1,3}$  & irred. unirat. \\\hline\hline

$9$ & $1$ & $\widehat{\E}_{9,1}$ & $8E_1+E_2$ & irred. unirat. \\\hline
$9$ & $2$ & $\widehat{\E}_{9,2}^{(I)}$ &  $4E_1+E_{1,2}$  & irred. unirat. \\
$9$ & $2$ & $\widehat{\E}_{9,2}^{(II)}$ &  $2(2E_1+E_2)$  & two  unirat. components
\\ \hline
$9$ & $3$ & $\widehat{\E}_{9,3}^{(I)}$ &  $2E_1+E_2+E_{1,2}$  & irred. unirat. \\
$9$ & $3$ & $\widehat{\E}_{9,3}^{(II)}$ &  $2E_1+2E_2+E_3$ & irred. unirat. \\ \hline 
$9$ & $4$ & $\widehat{\E}_{9,4}$ &  $2(E_1+E_{1,2})$ & two  unirat. components
\\ \hline\hline

$10$ & $1$ & $\widehat{\E}_{10,1}$ & $9E_1+E_2$ & irred. unirat. \\\hline
$10$ & $2$ & $\widehat{\E}_{10,2}$ &  $4E_1+E_2+E_3$  & irred. unirat. \\ \hline 
$10$ & $3$ & $\widehat{\E}_{10,3}^{(I)}$ &  $2E_1+E_2+E_3+E_4$  & irred. unirat. \\
$10$ & $3$ & $\widehat{\E}_{10,3}^{(II)}$ &   $3(E_1+E_2)$  & irred. unirat. \\  \hline 
$10$ & $4$ & $\widehat{\E}_{10,4}$ &  $2E_{1,2} + E_1+E_2$  &
irred. unirat. \\ \hline\hline

$11$ & $1$ & $\widehat{\E}_{11,1}$ & $10E_1+E_2$ & irred. unirat. \\\hline
$11$ & $2$ & $\widehat{\E}_{11,2}^{(I)}$ & $5E_1+E_{1,2}$ & irred. unirat. \\
$11$ & $2$ & $\widehat{\E}_{11,2}^{(II)}$ & $5E_1+2E_{2}$ & irred. unirat. \\\hline
$11$ & $3$ & $\widehat{\E}_{11,3}$ & $3E_1+E_{2}+E_{1,3}$ & irred. unirat. \\\hline
$11$ & $4$ & $\widehat{\E}_{11,4}$ & $E_1+E_{2}+E_{3}+E_4+E_5$ & irred. unirat. \\\hline \hline

$12$ & $1$ & $\widehat{\E}_{12,1}$ & $11E_1+E_2$ & irred. unirat. \\\hline
$12$ & $2$ & $\widehat{\E}_{12,2}$ & $5E_1+E_2+E_3$ & irred. unirat. \\\hline
$12$ & $3$ & $\widehat{\E}_{12,3}^{(I)}$ & $3E_1+2E_2+E_3$ & irred. unirat. \\
$12$ & $3$ & $\widehat{\E}_{12,3}^{(II)}$ &  $3E_1+E_2+E_{1,2}$ & irred. unirat. \\\hline
$12$ & $4$ & $\widehat{\E}_{12,4}$ & $2E_1+E_2+E_3+E_{1,4}$ &
irred. unirat. \\\hline\hline

$13$ & $1$ & $\widehat{\E}_{13,1}$ & $12E_1+E_2$ & irred. unirat. \\\hline
$13$ & $2$ & $\widehat{\E}_{13,2}^{(I)}$ & $6E_1+E_{1,2}$ & irred. unirat. \\
$13$ & $2$ & $\widehat{\E}_{13,2}^{(II)}$ & $2(3E_1+E_2)$ & two  unirat. components \\  \hline 
$13$ & $3$ & $\widehat{\E}_{13,3}^{(I)}$ & $3E_1+E_2+E_3+E_4$ & irred. unirat.  \\
$13$ & $3$ & $\widehat{\E}_{13,3}^{(II)}$ & $4E_1+3E_2$  & irred. unirat. \\ \hline 
$13$ & $4$ & $\widehat{\E}_{13,4}^{(I)}$ & $2E_1+2E_2+E_{1,2}$ & irred. unirat. \\
$13$ & $4$ & $\widehat{\E}_{13,4}^{(II)}$ & $2(E_1+E_2+E_3)$ & two  unirat. components \\
$13$ & $4$ & $\widehat{\E}_{13,4}^{(III)}$ &  $3E_1+2E_{1,2}$ & irred. unirat. \\\hline
\end{tabular}    
}
\end{minipage}
\hfill
\begin{minipage}[t]{\textwidth/2-1mm}
\begin{flushright}
\vspace{0mm}
\resizebox{\textwidth-4mm}{!}{%
\begin{tabular}{ |l | l |l|l|l|l|}
\hline			
$g$ &  $\phi$  & comp.  & dec. type &  $\rho^{-1}$    
\\\hline\hline

$14$ & $1$ & $\widehat{\E}_{14,1}$ & $13E_1+E_2$ & irred. unirat. \\\hline
$14$ & $2$ & $\widehat{\E}_{14,2}$ & $6E_1+E_2+E_3$ & irred. unirat. \\\hline
$14$ & $3$ & $\widehat{\E}_{14,3}$ & $4E_1+E_2+E_{1,3}$ & irred. unirat. \\\hline
$14$ & $4$ & $\widehat{\E}_{14,4}^{(I)}$ & $2E_1+2E_2+E_{3}+E_4$ & irred. unirat. \\
$14$ & $4$ & $\widehat{\E}_{14,4}^{(II)}$ & $3E_{1,2}+E_1+E_2$ & irred. unirat. \\\hline\hline

$15$ & $1$ & $\widehat{\E}_{15,1}$ & $14E_1+E_2$ & irred. unirat. \\\hline
$15$ & $2$ & $\widehat{\E}_{15,2}^{(I)}$ & $7E_1+E_{1,2}$ & irred. unirat. \\
$15$ & $2$ & $\widehat{\E}_{15,2}^{(II)}$ & $7E_1+2E_{2}$ & irred. unirat. \\\hline
$15$ & $3$ & $\widehat{\E}_{15,3}^{(I)}$ & $4E_1+2E_{2}+E_3$ & irred. unirat. \\
$15$ & $3$ & $\widehat{\E}_{15,3}^{(II)}$ & $4E_1+E_{2}+E_{1,2}$ & irred. unirat. \\\hline
$15$ & $4$ & $\widehat{\E}_{15,4}^{(I)}$ & $2E_1+E_{2}+E_3+E_4+E_5$ & irred. unirat. \\
$15$ & $4$ & $\widehat{\E}_{15,4}^{(II)}$ & $3E_1+2E_{2}+E_{1,3}$ & irred. unirat. \\\hline
$15$ & $5$ & $\widehat{\E}_{15,5}$ & $2E_1+E_{2}+2E_{1,2}$ & irred. unirat. \\\hline\hline

$16$ & $1$ & $\widehat{\E}_{16,1}$ & $15E_1+E_2$ & irred. unirat. \\\hline
$16$ & $2$ & $\widehat{\E}_{16,2}$ & $7E_1+E_2+E_3$ & irred. unirat. \\\hline
$16$ & $3$ & $\widehat{\E}_{16,3}^{(I)}$ & $4E_1+E_{2}+E_3+E_4$ & irred. unirat. \\
$16$ & $3$ & $\widehat{\E}_{16,3}^{(II)}$ & $5E_1+3E_{2}$ & irred. unirat. \\\hline
$16$ & $4$ & $\widehat{\E}_{16,4}^{(I)}$ & $3E_1+3E_{2}+E_3$ & irred. unirat. \\
$16$ & $4$ & $\widehat{\E}_{16,4}^{(II)}$ & $3E_1+E_{2}+E_3+E_{1,4}$ & irred. unirat. \\\hline
$16$ & $5$ & $\widehat{\E}_{16,5}$ & $E_1+E_{2}+E_3+E_{4}+E_5+E_6$ & irred. uniruled \\\hline\hline

$17$ & $1$ & $\widehat{\E}_{17,1}$ & $16E_1+E_2$ & irred. unirat. \\\hline
$17$ & $2$ & $\widehat{\E}_{17,2}^{(I)}$ & $8E_1+E_{1,2}$ & irred. unirat.\\
$17$ & $2$ & $\widehat{\E}_{17,2}^{(II)}$ & $2(4E_1+E_{2})$  & two  unirat. components \\\hline
$17$ & $3$ & $\widehat{\E}_{17,3}$ & $5E_1+E_2+E_{1,3}$ & irred. unirat. \\\hline
$17$ & $4$ & $\widehat{\E}_{17,4}^{(I)}$ & $3E_1+2E_2+2E_3$ & irred. unirat.\\
$17$ & $4$ & $\widehat{\E}_{17,4}^{(II)}$ & $3E_1+2E_2+E_{1,2}$  & irred. unirat.\\
$17$ & $4$ & $\widehat{\E}_{17,4}^{(III)}$ & $2(2E_1+E_{1,2})$ & two  unirat. components \\
$17$ & $4$ & $\widehat{\E}_{17,4}^{(IV)}$ &  $4(E_1+E_2)$ & two  unirat. components \\\hline
$17$ & $5$ & $\widehat{\E}_{17,5}$ & $2E_1+E_2+E_3+E_4+E_{1,5}$ & irred. uniruled  \\\hline\hline

$18$ & $1$ & $\widehat{\E}_{18,1}$ & $17E_1+E_2$ & irred. unirat. \\\hline
$18$ & $2$ & $\widehat{\E}_{18,2}$ & $8E_1+E_2+E_3$ &
irred. unirat. \\\hline 
$18$ & $3$ & $\widehat{\E}_{18,3}^{(I)}$ & $5E_1+2E_2+E_3$ &
irred. unirat.\\ 
$18$ & $3$ & $\widehat{\E}_{18,3}^{(II)}$ & $5E_1+E_2+E_{1,2}$  & irred. unirat.\\\hline 
$18$ & $4$ & $\widehat{\E}_{18,4}^{(I)}$ & $3E_1+2E_2+E_3+E_4$ & irred. unirat.\\
$18$ & $4$ & $\widehat{\E}_{18,4}^{(II)}$ & $4E_{1,2}+E_1+E_2$  & irred. unirat.\\\hline 
$18$ & $5$ & $\widehat{\E}_{18,5}^{(I)}$ & $3E_1+E_2+2E_{1,3}$ & irred. unirat.\\
$18$ & $5$ & $\widehat{\E}_{18,5}^{(II)}$ & $2E_1+2E_2+E_3+E_{1,2}$  & irred. unirat.\\\hline\hline 

$19$ & $1$ & $\widehat{\E}_{19,1}$ & $18E_1+E_2$ & irred. unirat. \\\hline
$19$ & $2$ & $\widehat{\E}_{19,2}^{(I)}$ & $9E_1+E_{1,2}$ & irred. unirat. \\
$19$ & $2$ & $\widehat{\E}_{19,2}^{(II)}$ & $9E_1+2E_2$ & irred. unirat. \\\hline
$19$ & $3$ & $\widehat{\E}_{19,3}^{(I)}$ & $5E_1+E_2+E_3+E_4$ & irred. unirat.\\
$19$ & $3$ & $\widehat{\E}_{19,3}^{(II)}$ & $3(2E_1+E_2)$  & irred. unirat.\\\hline 
$19$ & $4$ & $\widehat{\E}_{19,4}^{(I)}$ & $3E_1+E_2+E_3+E_4+E_5$ & irred. unirat.\\
$19$ & $4$ & $\widehat{\E}_{19,4}^{(II)}$ & $4E_1+2E_2+E_{1,3}$  & irred. unirat.\\\hline 
$19$ & $5$ & $\widehat{\E}_{19,5}^{(I)}$ & $2E_1+2E_2+2E_3+E_4$ & irred. unirat.\\
$19$ & $5$ & $\widehat{\E}_{19,5}^{(II)}$ & $3E_{1,2}+E_1+E_2+E_3$  & irred. unirat.\\\hline 
$19$ & $6$ & $\widehat{\E}_{19,6}$ & $3(E_1+E_{1,2})$ & irred. unirat. \\\hline\hline

$20$ & $1$ & $\widehat{\E}_{20,1}$ & $19E_1+E_2$ & irred. unirat. \\\hline
$20$ & $2$ & $\widehat{\E}_{20,2}$ & $9E_1+E_2+E_3$ & irred. unirat. \\\hline
$20$ & $3$ & $\widehat{\E}_{20,3}$ & $6E_1+E_2+E_{1,3}$ & irred. unirat. \\\hline

\end{tabular}    
}
\end{flushright}
\end{minipage}

\noindent
\begin{minipage}[t]{\textwidth/2-1mm}
\vspace{0mm}
\resizebox{\textwidth-4mm}{!}{%
\begin{tabular}{ |l | l |l|l|l|l|}
\hline			
$g$ &  $\phi$  & comp.  & dec. type &  $\rho^{-1}$    
\\\hline\hline

$20$ & $4$ & $\widehat{\E}_{20,4}^{(I)}$ & $4E_1+3E_2+E_3$ & irred. unirat.\\
$20$ & $4$ & $\widehat{\E}_{20,4}^{(II)}$ & $4E_1+E_2+E_3+E_{1,4}$  & irred. unirat.\\\hline 
$20$ & $5$ & $\widehat{\E}_{20,5}^{(I)}$ & $2E_1+2E_2+E_3+E_4+E_5$ & irred. unirat.\\
$20$ & $5$ & $\widehat{\E}_{20,5}^{(II)}$ & $3E_1+E_2+2E_{1,2}$  & irred. unirat.\\\hline\hline 

$21$ & $1$ & $\widehat{\E}_{21,1}$ & $20E_1+E_2$ & irred. unirat. \\\hline
$21$ & $2$ & $\widehat{\E}_{21,2}^{(I)}$ & $10E_1+E_{1,2}$ & irred. unirat. \\
$21$ & $2$ & $\widehat{\E}_{21,2}^{(II)}$ & $10E_1+2E_2$ & two  unirat. components \\\hline
$21$ & $3$ & $\widehat{\E}_{21,3}^{(I)}$ & $6E_1+E_2+E_{1,2}$ & irred. unirat.\\
$21$ & $3$ & $\widehat{\E}_{21,3}^{(II)}$ & $6E_1+2E_2+E_3$  & irred. unirat.\\\hline 

$21$ & $4$ & $\widehat{\E}_{21,4}^{(I)}$ & $5E_1+4E_2$ & irred. unirat.\\
$21$ & $4$ & $\widehat{\E}_{21,4}^{(II)}$ & $5E_1+2E_{1,2}$ & irred. unirat.\\
$21$ & $4$ & $\widehat{\E}_{21,4}^{(III)}$ & $4E_1+2E_2+2E_3$ &  two  unirat. components \\
$21$ & $4$ & $\widehat{\E}_{21,4}^{(IV)}$ & $4E_1+2E_2+E_{1,2}$  & irred. unirat.\\\hline 

$21$ & $5$ & $\widehat{\E}_{21,5}^{(I)}$ & $2E_1+E_2+E_3+E_4+E_5+E_6$ & ?? \\
$21$ & $5$ & $\widehat{\E}_{21,5}^{(II)}$ & $3E_{1}+2E_2+E_3+E_{1,4}$  & irred. unirat. \\\hline 

$21$ & $6$ & $\widehat{\E}_{21,6}$ & $2(E_1+E_2+E_{1,2})$ & two  unirat. components \\\hline\hline

$22$ & $1$ & $\widehat{\E}_{22,1}$ & $21E_1+E_2$ & irred. unirat. \\\hline
$22$ & $2$ & $\widehat{\E}_{22,2}$ & $10E_1+E_2+E_3$ & irred. unirat. \\\hline
$22$ & $3$ & $\widehat{\E}_{22,3}^{(I)}$ & $6E_1+E_2+E_3+E_4$ & irred. unirat. \\
$22$ & $3$ & $\widehat{\E}_{22,3}^{(II)}$ & $7E_1+3E_2$ & irred. unirat. \\\hline

$22$ & $4$ & $\widehat{\E}_{22,4}^{(I)}$ & $4E_1+2E_2+E_3+E_4$ & irred. unirat.\\
$22$ & $4$ & $\widehat{\E}_{22,4}^{(II)}$ & $5E_{1,2}+E_1+E_2$  & irred. unirat.\\\hline

$22$ & $5$ & $\widehat{\E}_{22,5}^{(I)}$ & $3E_1+E_2+E_3+E_4+E_{1,5}$ & irred. uniruled \\
$22$ & $5$ & $\widehat{\E}_{22,5}^{(II)}$ & $3E_1+3E_2+E_{1,2}$  & irred. unirat.\\
$22$ & $5$ & $\widehat{\E}_{22,5}^{(III)}$ & $3E_1+3E_2+2E_3$  & irred. unirat.\\\hline 
$22$ & $6$ & $\widehat{\E}_{22,6}$ & $E_1+E_2+E_3+E_4+E_5+E_6+E_7$  & irred. unirat.\\\hline\hline

$23$ & $1$ & $\widehat{\E}_{23,1}$ & $22E_1+E_2$ & irred. unirat. \\\hline
$23$ & $2$ & $\widehat{\E}_{23,2}^{(I)}$ & $11E_1+E_{1,2}$ & irred. unirat. \\
$23$ & $2$ & $\widehat{\E}_{23,2}^{(II)}$ & $11E_1+2E_2$ & irred. unirat.  \\\hline
$23$ & $3$ & $\widehat{\E}_{23,3}$ & $7E_1+E_2+E_{1,3}$ & irred. unirat.\\ \hline 

$23$ & $4$ & $\widehat{\E}_{23,4}^{(I)}$ & $5E_1+2E_2+E_{1,3}$ & irred. unirat.\\
$23$ & $4$ & $\widehat{\E}_{23,4}^{(II)}$ & $4E_1+E_2+E_3+E_4+E_5$ & irred. unirat.\\\hline 

$23$ & $5$ & $\widehat{\E}_{23,5}^{(I)}$ & $4E_1+E_2+2E_{1,3}$ & irred. unirat.  \\
$23$ & $5$ & $\widehat{\E}_{23,5}^{(II)}$ & $3E_{1}+2E_2+E_3+E_{1,2}$  & irred. unirat. \\
$23$ & $5$ & $\widehat{\E}_{23,5}^{(III)}$ & $3E_{1}+3E_2+E_3+E_4$  & irred. unirat.
\\\hline 

$23$ & $6$ & $\widehat{\E}_{23,6}$ & $2E_1+E_2+E_3+E_4+E_5+E_{1,6}$ & ?? \\\hline\hline

$24$ & $1$ & $\widehat{\E}_{24,1}$ & $23E_1+E_2$ & irred. unirat. \\\hline
$24$ & $2$ & $\widehat{\E}_{24,2}$ & $11E_1+E_2+E_3$ & irred. unirat. \\\hline
$24$ & $3$ & $\widehat{\E}_{24,3}^{(I)}$ & $7E_1+E_2+E_{1,2}$ & irred. unirat. \\
$24$ & $3$ & $\widehat{\E}_{24,3}^{(II)}$ & $7E_1+2E_2+E_3$ & irred. unirat. \\\hline

$24$ & $4$ & $\widehat{\E}_{24,4}^{(I)}$ & $5E_1+3E_2+E_3$ & irred. unirat.\\
$24$ & $4$ & $\widehat{\E}_{24,4}^{(II)}$ & $5E_1+E_2+E_3+E_{1,4}$  & irred. unirat.\\\hline 

$24$ & $5$ & $\widehat{\E}_{24,5}^{(I)}$ & $3E_1+2E_2+2E_3+E_4$ & irred. uniruled \\
$24$ & $5$ & $\widehat{\E}_{24,5}^{(II)}$ & $4E_{1,2}+E_1+E_2+E_3$  & irred. unirat.\\
$24$ & $5$ & $\widehat{\E}_{24,5}^{(III)}$ & $4E_1+3E_2+E_{1,3}$  & irred. unirat.\\\hline 
$24$ & $6$ & $\widehat{\E}_{24,6}^{(I)}$ & $3E_1+E_2+E_3+2E_{1,4}$  & irred. unirat.\\
$24$ & $6$ & $\widehat{\E}_{24,6}^{(II)}$ & $2E_1+2E_2+E_3+E_4+E_{1,2}$  & irred. uniruled \\\hline\hline 

$25$ & $1$ & $\widehat{\E}_{25,1}$ & $24E_1+E_2$ & irred. unirat. \\\hline
$25$ & $2$ & $\widehat{\E}_{25,2}^{(I)}$ & $12E_1+E_{1,2}$ & irred. unirat. \\
$25$ & $2$ & $\widehat{\E}_{25,2}^{(II)}$ & $2(6E_1+E_2)$ & two  unirat. components \\\hline
$25$ & $3$ & $\widehat{\E}_{25,3}^{(I)}$ & $7E_1+E_2+E_3+E_4$ & irred. unirat.\\
$25$ & $3$ & $\widehat{\E}_{25,3}^{(II)}$ & $8E_1+3E_2$  & irred. unirat.\\\hline 
$25$ & $4$ & $\widehat{\E}_{25,4}^{(I)}$ & $2(3E_1+2E_2)$ & two  unirat. components \\
$25$ & $4$ & $\widehat{\E}_{25,4}^{(II)}$ & $2(3E_1+E_{1,2})$ & two  unirat. components \\
$25$ & $4$ & $\widehat{\E}_{25,4}^{(III)}$ & $5E_1+2E_2+2E_3$ & irred. unirat.\\
$25$ & $4$ & $\widehat{\E}_{25,4}^{(IV)}$ & $5E_1+2E_2+E_{1,2}$  & irred. unirat.\\\hline 
$25$ & $5$ & $\widehat{\E}_{25,5}^{(I)}$ & $4E_1+E_2+2E_{1,2}$ & irred. unirat. \\
$25$ & $5$ & $\widehat{\E}_{25,5}^{(II)}$ & $3E_{1}+2E_2+E_3+E_4+E_5$  & irred. unirat. \\
$25$ & $5$ & $\widehat{\E}_{25,5}^{(III)}$ & $4E_{1}+4E_2+E_3$  & irred. unirat.
\\\hline 
$25$ & $6$ & $\widehat{\E}_{25,6}^{(I)}$ & $4E_1+3E_{1,2}$ & irred. unirat. \\\hline
\end{tabular}    
}
\end{minipage}
\hfill
\begin{minipage}[t]{\textwidth/2-1mm}
\vspace{0mm}
\begin{flushright}
\resizebox{\textwidth-4mm}{!}{%
\begin{tabular}{ |l | l |l|l|l|l|}
\hline			
$g$ &  $\phi$  & comp.  & dec. type &  $\rho^{-1}$    
\\\hline\hline
$25$ & $6$ & $\widehat{\E}_{25,6}^{(II)}$ & $2(E_1+E_2+E_3+E_4)$ & two  unirat. components \\
$25$ & $6$ & $\widehat{\E}_{25,6}^{(III)}$ & $3E_{1,2}+E_1+E_2+E_3+E_4$ & irred. uniruled  \\\hline\hline

$26$ & $1$ & $\widehat{\E}_{26,1}$ & $25E_1+E_2$ & irred. unirat. \\\hline
$26$ & $2$ & $\widehat{\E}_{26,2}$ & $12E_1+E_2+E_3$ & irred. unirat. \\\hline
$26$ & $3$ & $\widehat{\E}_{26,3}$ & $8E_1+E_2+E_{1,3}$ & irred. unirat. \\\hline
$26$ & $4$ & $\widehat{\E}_{26,4}^{(I)}$ & $5E_1+2E_2+E_3+E_4$ & irred. unirat.\\
$26$ & $4$ & $\widehat{\E}_{26,4}^{(II)}$ & $6E_{1,2}+E_1+E_2$  &
irred. unirat.\\\hline 
$26$ & $5$ & $\widehat{\E}_{26,5}^{(I)}$ & $3E_1+E_2+E_3+E_4+E_{5}+E_6$ & ?? \\
$26$ & $5$ & $\widehat{\E}_{26,5}^{(II)}$ & $4E_1+2E_2+E_3+E_{1,4}$  & irred. unirat.\\
$26$ & $5$ & $\widehat{\E}_{26,5}^{(III)}$ & $5(E_1+E_2)$  & irred. unirat.\\\hline 

$26$ & $6$ & $\widehat{\E}_{26,6}^{(I)}$ & $3E_1+E_2+E_3+2E_{1,2}$  & irred. unirat.\\
$26$ & $6$ & $\widehat{\E}_{26,6}^{(II)}$ & $2E_1+2E_2+2E_3+E_4+E_5$  & irred. unirat.\\\hline\hline 

$27$ & $1$ & $\widehat{\E}_{27,1}$ & $26E_1+E_2$ & irred. unirat. \\\hline
$27$ & $2$ & $\widehat{\E}_{27,2}^{(I)}$ & $13E_1+E_{1,2}$ & irred. unirat. \\
$27$ & $2$ & $\widehat{\E}_{27,2}^{(II)}$ & $13E_1+2E_2$ & irred. unirat. \\\hline
$27$ & $3$ & $\widehat{\E}_{27,3}^{(I)}$ & $8E_1+E_2+E_{1,2}$ & irred. unirat.\\ 
$27$ & $3$ & $\widehat{\E}_{27,3}^{(II)}$ & $8E_1+2E_2+E_{3}$ & irred. unirat.\\ \hline 

$27$ & $4$ & $\widehat{\E}_{27,4}^{(I)}$ & $6E_1+2E_2+E_{1,3}$ & irred. unirat.\\
$27$ & $4$ & $\widehat{\E}_{27,4}^{(II)}$ & $5E_1+E_2+E_3+E_4+E_5$ & irred. unirat.\\\hline 

$27$ & $5$ & $\widehat{\E}_{27,5}^{(I)}$ & $4E_1+E_2+E_3+E_4+E_{1,5}$ & irred. uniruled   \\
$27$ & $5$ & $\widehat{\E}_{27,5}^{(II)}$ & $4E_{1}+3E_2+E_{1,2}$  & irred. unirat. \\
$27$ & $5$ & $\widehat{\E}_{27,5}^{(III)}$ & $4E_{1}+3E_2+2E_3$  & irred. unirat.
\\\hline 

$27$ & $6$ & $\widehat{\E}_{27,6}^{(I)}$ & $3E_1+2E_2+2E_{1,2}$ & irred. unirat.  \\
$27$ & $6$ & $\widehat{\E}_{27,6}^{(II)}$ & $3E_{1}+2E_2+2E_3+E_{1,4}$  & irred. unirat. \\
$27$ & $6$ & $\widehat{\E}_{27,6}^{(III)}$ & $2E_1+2E_2+E_3+E_4+E_5+E_{6}$ & ?? \\\hline\hline

$28$ & $1$ & $\widehat{\E}_{28,1}$ & $27E_1+E_2$ & irred. unirat. \\\hline
$28$ & $2$ & $\widehat{\E}_{28,2}$ & $13E_1+E_2+E_3$ & irred. unirat. \\\hline
$28$ & $3$ & $\widehat{\E}_{28,3}^{(I)}$ & $8E_1+E_2+E_3+E_4$ & irred. unirat. \\
$28$ & $3$ & $\widehat{\E}_{28,3}^{(II)}$ & $3(3E_1+E_2)$ & irred. unirat. \\\hline

$28$ & $4$ & $\widehat{\E}_{28,4}^{(I)}$ & $6E_1+3E_2+E_3$ & irred. unirat.\\
$28$ & $4$ & $\widehat{\E}_{28,4}^{(II)}$ & $6E_1+E_2+E_3+E_{1,4}$  & irred. unirat.\\\hline 

$28$ & $5$ & $\widehat{\E}_{28,5}^{(I)}$ & $5E_1+E_2+2E_{1,3}$ & irred. unirat.\\
$28$ & $5$ & $\widehat{\E}_{28,5}^{(II)}$ & $4E_1+2E_2+E_3+E_{1,2}$  & irred. unirat.\\
$28$ & $5$ & $\widehat{\E}_{28,5}^{(III)}$ & $4E_1+3E_2+E_3+E_4$  & irred. unirat.\\\hline 

$28$ & $6$ & $\widehat{\E}_{28,6}^{(I)}$ & $2E_1+E_2+E_3+E_4+E_5+E_6+E_7$  &  irred. uniruled  \\
$28$ & $6$ & $\widehat{\E}_{28,6}^{(II)}$ & $3(E_1+E_2+E_3)$  & irred. unirat.\\
$28$ & $6$ & $\widehat{\E}_{28,6}^{(III)}$ & $3E_1+2E_2+E_3+E_4+E_{1,5}$  & irred. uniruled \\\hline 
$28$ & $7$ & $\widehat{\E}_{28,7}$ & $3E_1+E_2+3E_{1,2}$  & irred. unirat.\\\hline\hline 

$29$ & $1$ & $\widehat{\E}_{29,1}$ & $28E_1+E_2$ & irred. unirat. \\\hline
$29$ & $2$ & $\widehat{\E}_{29,2}^{(I)}$ & $14E_1+E_{1,2}$ & irred. unirat. \\
$29$ & $2$ & $\widehat{\E}_{29,2}^{(II)}$ & $2(7E_1+E_2)$ & two  unirat. components \\\hline
$29$ & $3$ & $\widehat{\E}_{29,3}$ & $9E_1+E_2+E_{1,3}$ & irred. unirat. \\\hline 

$29$ & $4$ & $\widehat{\E}_{29,4}^{(I)}$ & $7E_1+4E_2$ & irred. unirat.\\
$29$ & $4$ & $\widehat{\E}_{29,4}^{(II)}$ & $7E_1+2E_{1,2}$ & irred. unirat.\\
$29$ & $4$ & $\widehat{\E}_{29,4}^{(III)}$ & $2(3E_1+E_{2}+E_3)$ & two  unirat. components \\
$29$ & $4$ & $\widehat{\E}_{29,4}^{(IV)}$ & $6E_1+2E_2+E_{1,2}$  & irred. unirat.  \\\hline 

$29$ & $5$ & $\widehat{\E}_{29,5}^{(I)}$ & $4E_1+2E_2+2E_3+E_4$ & irred. unirat. \\
$29$ & $5$ & $\widehat{\E}_{29,5}^{(II)}$ & $5E_{1,2}+E_1+E_2+E_3$  & irred. unirat. \\
$29$ & $5$ & $\widehat{\E}_{29,5}^{(III)}$ & $5E_{1}+3E_2+E_{1,3}$  & irred. unirat.
\\\hline 

$29$ & $6$ & $\widehat{\E}_{29,6}^{(I)}$ & $3E_1+E_2+E_3+E_4+E_5+E_{1,6}$ & ?? \\
$29$ & $6$ & $\widehat{\E}_{29,6}^{(II)}$ & $2(2E_1+E_2+E_{1,3})$ & two  unirat. components \\
$29$ & $6$ & $\widehat{\E}_{29,6}^{(III)}$ & $3E_1+3E_2+E_3+E_{1,2}$ & irred. unirat. \\\hline\hline

$30$ & $1$ & $\widehat{\E}_{30,1}$ & $29E_1+E_2$ & irred. unirat. \\\hline
$30$ & $2$ & $\widehat{\E}_{30,2}$ & $14E_1+E_2+E_3$ & irred. unirat. \\\hline
$30$ & $3$ & $\widehat{\E}_{30,3}^{(I)}$ & $9E_1+2E_2+E_{3}$ & irred. unirat. \\
$30$ & $3$ & $\widehat{\E}_{30,3}^{(II)}$ & $9E_1+E_2+E_{1,2}$ & irred. unirat. \\\hline

$30$ & $4$ & $\widehat{\E}_{30,4}^{(I)}$ & $6E_1+2E_2+E_3+E_4$ & irred. unirat.\\
$30$ & $4$ & $\widehat{\E}_{30,4}^{(II)}$ & $7E_{1,2}+E_1+E_2$  & irred. unirat.\\\hline 
$30$ & $5$ & $\widehat{\E}_{30,5}^{(I)}$ & $5E_1+E_2+2E_{1,2}$ & irred. unirat. \\
$30$ & $5$ & $\widehat{\E}_{30,5}^{(II)}$ & $4E_1+2E_2+E_3+E_4+E_5$  & irred. unirat.\\
$30$ & $5$ & $\widehat{\E}_{30,5}^{(III)}$ & $5E_1+4E_2+E_3$  & irred. unirat.\\\hline 

$30$ & $6$ & $\widehat{\E}_{30,6}^{(I)}$ & $4E_1+E_2+E_3+2E_{1,4}$  & irred. unirat.\\
$30$ & $6$ & $\widehat{\E}_{30,6}^{(II)}$ & $3E_1+2E_2+E_3+E_4+E_{1,2}$  & irred. unirat.\\
$30$ & $6$ & $\widehat{\E}_{30,6}^{(III)}$ & $3E_1+3E_2+2E_3+E_4$  & irred. unirat.\\
$30$ & $6$ & $\widehat{\E}_{30,6}^{(IV)}$ & $4E_{1,2}+E_1+E_2+2E_3$  & irred. unirat.\\\hline 

$30$ & $7$ & $\widehat{\E}_{30,7}^{(I)}$ & $2E_1+E_2+E_3+E_4+E_5+E_6+E_{1,7}$ & irred. unirat. (cf. Rem. \ref{rem:notunique})\\
$30$ & $7$ & $\widehat{\E}_{30,7}^{(II)}$ & $2E_1+4E_2+E_3+E_4+E_5$  & irred. unirat. \\\hline
\end{tabular}    
}
\end{flushright}
\end{minipage}

\end{document}